\documentclass[12pt,a4paper,twoside]{article}

\usepackage{fancyhdr}
\usepackage{amsmath, amsfonts, amssymb, amsthm}
\usepackage[english]{babel}
\usepackage[all]{xy}
\usepackage{graphicx}
\usepackage[T1]{fontenc}
\usepackage[utf8]{inputenc}
\usepackage{color}
\usepackage[a4paper,top=3cm,bottom=3cm,left=3cm,right=3cm]{geometry}
\usepackage{tikz}
\usetikzlibrary{matrix, arrows, decorations.pathmorphing}
\usepackage{plain}

\title{\textbf{Linearization of monomial ideals}}
\author{Milo Orlich\footnote{Department of Mathematics and Systems Analysis, Aalto University, Espoo, Finland.\newline Email: \tt milo.orlich@aalto.fi}}
\newtheoremstyle{break}%
{}{}%
{\slshape}{}%
{\bfseries}{.}% 
{5pt}{}

\theoremstyle{definition}
\newtheorem{definition}{Definition}[section]
\newtheorem{fundefinition}[definition]{Functorial Definition}
\newtheorem*{definition*}{Definition}
\newtheorem{example}[definition]{Example}

\newtheorem{notation}[definition]{Notation}
\newtheorem{question}[definition]{Question}

\newtheorem{remark}[definition]{Remark}

\theoremstyle{break}
\newtheorem{lemma}[definition]{Lemma}
\newtheorem{proposition}[definition]{Proposition}
\newtheorem*{proposition*}{Proposition}
\newtheorem{theorem}[definition]{Theorem}
\newtheorem*{theorem*}{Theorem}
\newtheorem*{maintheorem*}{Main Theorem}
\newtheorem{corollary}[definition]{Corollary}

\newtheorem*{conjecture*}{Conjecture}

\newcommand\tbs[1]{\textsl{\textbf{#1}}}  % usato per i termini definiti

\def\K{\mathbb K}
\def\N{\mathbb N}

\def\Z{\mathbb Z}

\def\EE{\mathcal E}

\def\G{\mathbf G}
\def\F{\mathbf F}

\def\mm{\mathfrak m}

\def\ii{\mathbf i}
\def\jj{\mathbf j}

\def\la{\longrightarrow}
\def\map{\longmapsto}
\def\:{\colon}

\def\lcm{\mathrm{lcm}}

\def\depth{\mathrm{depth}}
\def\pd{\mathrm{projdim}}

\def\mult{\mathrm{mult}}  %% molteplicità

\def\dist{\mathrm{dist}}
\def\diam{\mathrm{diam}}

\def\Syz{\mathrm{Syz}}

\def\Tor{\mathrm{Tor}}

\def\Supp{\mathrm{Supp}}

\def\de{\partial}

\def\se{\subseteq}

\def\vep{\varepsilon}
\def\lam{\lambda}
\def\cl{\overline}

\def\iso{\cong}

\def\emse{\emptyset}

\def\cocoa{\mbox{\rm 
   C\kern-.13em o\kern-.07 em C\kern-.13em o\kern-.15em A}}

\def\hom{\mathrm{hom}}

\def\Lex{\mathrm{Lex}}

\def\Lin{\mathrm{Lin}}

\def\LIN{\Lin^*}

\def\eq{\mathrm{eq}}
\def\sqf{\mathrm{sqf}}

\def\Ideals{\mathbf{Ideals}}
\def\MonIdeals{\mathbf{MonIdeals}}

\def\sign{\mathrm{sign}}

\def\RC{\mathrm{RC}}

\def\til{~}

\begin{document}

\maketitle

\abstract{We introduce a construction, called \emph{linearization}, that associates to any monomial ideal $I\subset\K[x_1,\dots,x_n]$ an ideal $\Lin(I)$ in a larger polynomial ring. The main feature of this construction is that the new ideal $\Lin(I)$ has linear quotients. In particular, since $\Lin(I)$ is generated in a single degree, it follows that $\Lin(I)$ has a linear resolution. We investigate some properties of this construction, such as its interplay with classical operations on ideals, its Betti numbers, functoriality and combinatorial interpretations. We moreover introduce an auxiliary construction, called \emph{equification}, that associates to an arbitrary monomial ideal $J$ an ideal $J^\eq$, generated in a single degree, in a polynomial ring with one more variable. We study some of the homological and combinatorial properties of the equification, which can be seen as a monomial analogue of the well-known homogenization construction.}

\tableofcontents 

\section{Introduction}

Among all ideals, those with a linear resolution are somehow ``simpler'' and have a vast literature. There are already ways of constructing ideals (or more in general modules) with linear resolutions (see for instance~\cite{EiGo}). The goal of this paper is to introduce and study a new construction, called \emph{linearization}, which converts an arbitrary monomial ideal~$I$ in a monomial ideal~$\Lin(I)$ with a linear resolution. In fact, $\Lin(I)$ has an even stronger property: it has linear quotients.
We start this introduction with an overview of monomial ideals and their resolutions. We then move on to outline the content of this paper.

\subsubsection*{Monomial ideals and their free resolutions}

Let $S=\K[x_1,\dots,x_n]$ be a polynomial ring over a field $\K$. A \emph{monomial ideal} is an ideal of $S$ generated by monomials. Such ideals are a core object of study in commutative algebra because, thanks to their highly combinatorial structure, they are easier to understand than arbitrary ideals.  For a general reference about monomial ideals, see the monograph~\cite{HH}. 

One can often tackle problems about general ideals by studying monomial ideals instead, through the methods  of Gr\"obner bases (see for instance Chapter~2 of~\cite{HH} or Chapter~15 of~\cite{EisCA}): one associates to an arbitrary ideal $I$ its initial ideal, which is monomial and hence easier to deal with.

A \emph{free resolution} over $S$ of an ideal $I\subset S$ consists of a sequence~$F_\bullet=(F_i)_{i\in\N}$ of free $S$-modules and homomorphisms $d_i\:F_i\to F_{i-1}$, with $i\ge1$, such that there is a surjective homomorphism $\vep\:F_0\to I$ and such that $F_\bullet\stackrel\vep\to I\to0$ is an exact complex. If one assumes the resolution to be \emph{minimal} and \emph{graded} (so that the free modules have to be shitfted accordingly), one may write
$$F_i=\bigoplus_{j\in\Z}S(-j)^{\beta_{ij}}$$
and the natural numbers $\beta_{ij}$, called the \emph{graded Betti numbers of $I$}\/, are uniquely determined. 

In the 1960's Kaplansky posed the problem of studying systematically the resolutions of monomial ideals. A great deal of research has been done ever since, starting with Taylor's PhD thesis~\cite{Tay}, where she defined a canonical construction that works for every monomial ideal but gives a highly non-minimal resolution in general (see Section~26 of~\cite{Pe}). Despite the apparently easy structure of monomial ideals, their \emph{minimal} resolutions have been quite elusive for more than half a century. Many constructions have been introduced, that either provide minimal resolutions only for certain classes of monomial ideals, or complexes that work in great generality but are not always minimal resolutions, or even resolutions (see for instance Sections~14 and~28 of~\cite{Pe}). Some of these constructions provide fruitful combinatorial or topological  interpretations of what happens on the algebraic side. In particular, Hochster, Stanley and Reisner (see~\cite{Stan}, \cite{Rei} and Chapter~5 of~\cite{BH93}) introduced a machinery where one associates bijectively squarefree monomial ideals to simplicial complexes, via the so-called \emph{Stanley--Reisner correspondence}, which allows to understand free resolutions in terms of simplicial homology. This beautiful bridge between commutative algebra and combinatorial topology is an instance of what has been the general trend in the last decades until now: quoting Peeva's words in~\cite{Pe}, \emph{``introduce new ideas and constructions which either have strong applications or/and are beautiful''}. In~2019, Eagon, Miller and Ordog described a canonical  minimal resolution for every monomial ideal, in~\cite{EMO}.

\subsubsection*{Linearization of monomial ideals}

Recall that a monomial ideal $I$ has a \emph{$d$-linear resolution} if $\beta_{i,j}(I)=0$ for $j\ne i+d$. In particular, the only $j$ for which we can have $\beta_{0,j}(I)\ne0$ is $d$, meaning that all the minimal generators of~$I$ have degree $d$. For what concerns the higher homological positions, having a linear resolution means that when we fix bases and write the maps in the resolutions as matrices, the non-zero entries of those matrices are linear forms.

A well-known numerical invariant of any ideal (or module, in general), is its \emph{(Castelnuovo--Mumford) regularity}, which measures how complicated the resolution is. For an ideal with all generators of degree $d$, the regularity is equal to~$d$, the smallest possible, if and only if the resolution is $d$-linear. 

 Among the many papers concerning linear resolutions we refer first of all to the work~\cite{EiGo} of~{Eisenbud} and~{Goto} and the classical work~\cite{Steu} by Steurich. More recent directions of research involve families of ideals such that every product of  elements in the family has a linear resolution (see for instance~\cite{BrCo}).  We refer to~\cite{Pe} and ~\cite{HH} for more information, and we provide additional references in the main body of the paper.

\begin{definition*}
Let $I\subset\K[x_1,\dots,x_n]$ be a monomial ideal with minimal set of monomial generators $G(I)=\{f_1,\dots,f_m\}$, such that $f_1,\dots,f_m$ all have the same degree~$d$. For all $i\in\{1,\dots,n\}$, denote by $M_i$ the largest exponent with which~$x_i$ occurs in $G(I)$.
The \emph{linearization of $I$}, inside the polynomial ring $R:=\K[x_1,\dots,x_n,y_1,\dots,y_m]$, is the ideal 
\begin{align*}
\Lin(I)&:=\big(x_1^{a_1}\cdots x_n^{a_n}\mid a_1+\dots+a_n=d\,\text{ and $a_i\le M_i$ for all $i$}\big)\\
&\ \quad+\big(f_jy_j/x_k \mid \text{$x_k$ divides $f_j$},\ k=1,\dots,n,\ j=1,\dots,m\big).
\end{align*}
We call the first summand \emph{complete part of $\Lin(I)$} and the second summand \emph{last part of $\Lin(I)$}.
\end{definition*}

Observe that the ideal $\Lin(I)$ is generated in the same degree $d$ as the original ideal $I$. One of the main properties of $\Lin(I)$, and the reason for the name ``linearization'', is the following:
\begin{theorem*}[Corollary\til\ref{Linlinres}]
The ideal $\Lin(I$) has a linear resolution over $R$.
\end{theorem*}

This is actually implied by the fact that $\Lin(I)$ has ``linear quotients'', a property which we now recall. Given an ideal $I\subset S:=\K[x_1,\dots,x_n]$ and a polynomial $g\in S$, the \emph{colon ideal} (or \emph{quotient ideal}) of $I$ with respect to $g$ is 
$$I:g=\{f\in S\mid fg\in I\}.$$
%(We roughly recall the geometric meaning of this, in characteristic~{zero}: taking the variety associated to the  colon ideal $I:J$ correponds to taking the Zariski closure of the difference $V-W$, where $V$ is the variety associated to $I$ and $W$ to $J$; see~\cite{CLO} for more on this.) 
An ideal $I\subset S$ is said to have \emph{linear quotients} if there exists a system of generators $g_1,\dots,g_m$ of $I$ such that each colon ideal~$(g_1,\dots,g_{k-1}):g_k$ is generated by linear forms, for any~$k\in\{2,\dots,m\}$.  This depends in general on the order of $g_1,\dots,g_m$. For the combinatorial meaning of linear quotients and additional information, see Section\til8.2 of\til\cite{HH} and the next sections in this introduction. The main result of the paper,   which implies the fact that~$\Lin(I)$ has a linear resolution, is the following:

\begin{maintheorem*}[Theorem\til\ref{mainthm}]
Assume that the generators $f_1,\dots,f_m$ of $I$ are in decreasing lexicographic order.
List the generators of the complete part of $\Lin(I)$ in decreasing lexicographic order.  List the generators $\frac{f_j}{x_k}y_j$  of the last part first by increasing\til$j$, and secondly by increasing~$k$.
The ideal $\Lin(I)$  has linear quotients with respect to the given ordering of the generators.
\end{maintheorem*}

\subsubsection*{Outline of the paper}

In Section~\ref{randomtools} we recall the necessary backgroud and we prove the following:
\begin{proposition*}[Proposition~\ref{cropparequozientilinear}]
Let $I=(f_1,\dots,f_s)\subset \K[x_1,\dots,x_n]$ be a monomial ideal with linear quotients with respect to $f_1,\dots,f_s$. Fix a vector $v=(v_1,\dots,v_n)\in\N^n$ of non-negative integers and denote  by $I_{\le v}$ the ideal generated by  the generators $f_j=x_1^{a_1}\cdots x_n^{a_n}$ of~$I$ such that $a_i\le v_i$ for all\til$i$. Denote these generators as $f_{b_1},\dots,f_{b_t}$, with $b_1<\dots<b_t$. Then $I_{\le v}$ has linear quotients with respect to $f_{b_1},\dots,f_{b_t}$.
\end{proposition*}

In Section~\ref{lineargeneral} we define the linearization for ideals generated in a single degree and investigate some of its properties.%, including the Main Theorem above. 
We also define a slightly different construction, the $*$-linearization $\LIN(I)$, which is easier or more meaningful to consider for some results. We give a ``more precise'' functorial  definition of linearization in Section~\ref{sectionwithfunctor}, and we prove the two following results:
\begin{theorem*}[Theorem\til\ref{carattpolymatr}]
For a monomial ideal $I\subset S=\K[x_1,\dots,x_n]$ generated in degree~$d$, in the following cases $\Lin(I)$ is polymatroidal:
\begin{enumerate}
\item[(a)] $d=1$, that is, $I$ is generated by variables; 
\item[(b)] $d$ is arbitrary and $I$ is principal.
\end{enumerate}
In all other cases $\Lin(I)$ is not polymatroidal.
\end{theorem*}

\begin{theorem*}[Theorem~\ref{radhaslinearquotients}]
The radical ideal of $\LIN(I)$ has linear quotients, and hence  linear resolution.
\end{theorem*}

In Section~\ref{linearsquarefree} we focus on the case where $I$ is squarefree, which happens if and only if~$\Lin(I)$ is squarefree. This section is done mostly with $\LIN(I)$ instead of $\Lin(I)$, in order to make the notation less heavy. The difference between the two, in the squarefree case, is anyway not meaningful. We compute the Betti numbers of $\LIN(I)$, which depend only on how the monomials of degree~$d-1$ divide the minimal generators of~$I$. In ``simplicial terminology'', these monomials are called \emph{codimension-$1$~{facets}} of the simplicial complex whose facet ideal is~$I$. Since we have applications to hypergraphs in mind, we call them \emph{$(d-1)$-edges} instead.
In Corollary~\ref{bnumbersforlin} we get the Betti numbers of $\LIN(I)$ as follows. Let us call a \emph{$j$-cluster} a set of cardinality $j$ consisting of generators of $I$ that are divided by a same monomial of degree~$d-1$. For each\til$j$, denote by~$C_j$ the number of maximal $j$-clusters, that is, $j$-clusters that are not part of a $(j+1)$-cluster. A closed formula for the Betti numbers is then
\begin{align*}
\beta_i\big(\LIN(I)\big)&=\binom{i+d-1}{d-1}\binom n{i+d}
+\binom{n-d+1}i\bigg(md-\sum_{j\ge2}(j-1)C_j\bigg)\\
&\quad+\sum_{j\ge2}C_j\sum_{k=2}^j\binom{n-d+k}i,
\end{align*}
where $m$ is the number of generators of $I$ and $d$ is their degree. In particular, if we denote $N:=\max\{j\mid C_j\ne 0\}$, then we have
$$\pd_R(\LIN(I))=n-d+N.$$
In Section~\ref{gunnarsproof} we give an alternative, more conceptual proof that $\LIN(I)$ has a linear resolution and that its Betti numbers only depend on how the monomials of degree~$d-1$ divide the generators of~$I$. This second proof was  taught to me by Gunnar Fl\o ystad. We conclude in Section~\ref{hyperlin} with an interpretation of the squarefree linearization by means of hypergraphs, related to the work of  H\`a and Van Tuyl in~\cite{HV}. 

In Section~\ref{linnonequig} we introduce an auxiliary construction in order to generalize the linearization to arbitrary monomial ideals:
\begin{definition*}
Let $I$ be an arbitrary monomial ideal in $\K[x_1,\dots,x_n]$, with minimal system of monomial generators $G(I)=\{f_1,\dots,f_m\}$. Denote $d_j:=\deg(f_j)$ for all $j$ and $d:=\max\{d_j\mid j=1,\dots,m\}$. We define the \emph{equification of $I$}as
$$I^\eq:=(f_1z^{d-d_1},f_2z^{d-d_2},\dots,f_mz^{d-d_m})$$
in the polynomial ring $\K[x_1,\dots,x_n,z]$ with one extra variable $z$. We moreover define the \emph{linearization of $I$} as
$$\Lin(I^\eq),$$
where $\Lin$ is the linearization defined earlier for ideals generated in a single degree.
\end{definition*}
The equification construction seems to be interesting in its own right, and in Section~\ref{omogeneizzmonomials} we investigate some of its properties. For instance, in Proposition~\ref{inequalitybnumbs} we show that  the total Betti numbers of $I$ and $I^\eq$ satisfy the inequalities
 $$\beta_i(I)\le\beta_i(I^\eq)\qquad\text{for all $i>0$}$$
and $\beta_0(I)=\beta_0(I^\eq)$, where the resolutions are taken over the respective polynomial rings. Lastly, in Section~\ref{lcmlattices} we compare~$I$ and~$I^\eq$ by observing that the $\lcm$-lattice of~$I$ can be embedded in that of $I^\eq$.

In Section~\ref{domandeaperte} we discuss some open questions and possible future developments. In particular, it might be possible to generalize in some meaningful way the constructions of linearization and equification to more general (at least homogeneous) ideals.

\medskip

\noindent
\emph{Rees algebras.} We conclude this part of the introduction by drawing the reader's attention to a well-known construction, the Rees algebra of $I$. Although it will never be used in the paper, the reason for mentioning this is that in the definition of linearization the ring is enlarged by introducing new variables, and this might seem a bit artificial on one hand, or very similar to what one does when defining Rees algebras on the other hand. 
Let $I=(f_1,\dots,f_m)\subset S:=\K[x_1,\dots,x_m]$ be a homogeneous ideal, where $f_1,\dots,f_m$ are a minimal system of homogeneous generators. The \emph{Rees algebra of $I$} is the image of the $S$-algebra homomorphism
\begin{align*}
S[y_1,\dots,y_m]&\longrightarrow S[t]\\
y_i&\longmapsto f_it,
\end{align*}
namely the subalgebra $S[f_1t,\dots,f_mt]\subset S[t]$. The similarity with the linearization consists in the introduction of new variables $y_j$, as many as generators of the ideal. The similarity seems to end here: in particular, $\Lin(I)$ has its fundamental property  of having linear resolution. %is not shared by the Rees algebra of~$I$ in general. 
However, it would be interesting to investigate in the future if there is an analogous theory of deformations for $\Lin(I)$ as there is for the Rees algebra. See Section~{6.5} of~\cite{EisCA} for more about this.

% \emph{Polarizations.} The polarization construction takes a monomial ideal $I$ and gives a new ideal~$I^{\mathrm{pol}}$, called \emph{polarization of $I$}, which is \emph{squarefree} and lives in a larger polynomial ring. This does not change the homological properties of the ideal, but has the advantage that one has the machinery  of simplicial complexes to attack squarefree monomial ideals. The idea of polarization is that if a generator of $I$ contains a high power of some variables, we ``split'' that power in the product of different new variables. For instance, if a generator of $I\subset\K[x,y,z]$ is $x^2y^3z$, then the corresponding generator of the polarization $I^{\mathrm{pol}}\subset\K[x_1,x_2,x_3,y_1,y_2,z]$ will be
%$$x_1x_2x_3y_1y_2z,$$
%where we replace $x$ by new variables $x_1$, $x_2$ and $x_3$, and replace $y$ by $y_1$ and $y_2$. If we identify $x_1=x_2=x_3$ as a single variable $x$ and $y_1=y_2$ as $y$, we get back the original generator of~$I$.
%Polarizations of monomial ideals were first used by Hartshorne in his proof of the connectedness of the Hilbert scheme, see Chapter~4 of Hartshorne's paper~\cite{Hart}. In the last few years  a generalization of 

\subsubsection*{Motivation and combinatorial interpretations}

The contents of this section, in particular about Booth--Lueker graphs, are not necessary for the rest of the paper.  The purpose of this section is just to provide motivation for the linearization construction, and possible combinatorial applications.

In~1975 Booth and Lueker introduced a construction defined as follows, in their paper~\cite{BL75}.

\begin{definition}\label{defBL}
Let $G$ be a finite simple graph with vertices $x_1,\dots,x_n$. Let $e_1,\dots,e_m$ be the edges of~$G$. We define the \emph{Booth--Lueker graph of $G$}, denoted~$BL(G)$, on the set of vertices $\{x_1,\dots,x_n\}\cup\{y_1,\dots,y_m\}$, as follows:  for all $i$ and $j$, $BL(G)$ has the edge~$x_ix_j$, and for each edge $e_i=x_{i_1}x_{i_2}$ of $G$, $BL(G)$ has the edges~$x_{i_1}y_i$ and~$x_{i_2}y_i$.
\end{definition}

The original purpose of this was related to the graph isomorphism problem for graphs. This construction was studied from an algebraic point of view by the author of this paper et al.\til{}in~\cite{EJO}. See Figure~\ref{fig:BL} for an example of graph $G$ and its Booth--Lueker graph~$BL(G)$.
\begin{figure}
\begin{center}
\begin{tikzpicture} [>=latex,scale=1.1]
%\draw [help lines] (-4,-1) grid (2,2);
% original graph
% vertices
\fill (-4,0) circle (0.1);
\coordinate [label=below left: $x_4$] (x4) at (-4,0);
\fill (-4,1) circle (0.1);
\coordinate [label=above left: $x_1$] (x1) at (-4,1);
\fill (-3,0) circle (0.1);
\coordinate [label=below right: $x_3$] (x3) at (-3,0);
\fill (-3,1) circle (0.1);
\coordinate [label=above right: $x_2$] (x2) at (-3,1);
% edges
\draw [thick] (-3,1) -- (-3,0);
\coordinate [label=right: $e_2$] (e2) at (-3.05,.5);
\draw [thick] (-3,1) -- (-4,1);
\coordinate [label=above: $e_1$] (e1) at (-3.5,0.95);
\draw [thick] (-4,0) -- (-3,0);
\coordinate [label=below: $e_3$] (e3) at (-3.5,.05);
% BL graph
% vertices
\fill (0,0) circle (0.1);
\fill (0,1) circle (0.1);
\fill (1,0) circle (0.1);
\fill (1,1) circle (0.1);
\fill (2,1.5) circle (0.1);
\coordinate [label=right: $y_1$] (y1) at (2,1.5);
\fill (2,.5) circle (0.1);
\coordinate [label=right: $y_2$] (y2) at (2,.5);
\fill (2,-.5) circle (0.1);
\coordinate [label=right: $y_3$] (y3) at (2,-.5);
% edges
% the edges in the original graph
\draw [thick] (0,0) -- (1,0);
\draw [thick] (1,1) -- (0,1);
\draw [thick] (1,1) -- (1,0);
% the edges that complete the original graph
\draw [thick] (0,0) -- (0,1);
\draw [thick] (0,0) -- (1,1);
\draw [thick] (1,0) -- (0,1);
% the edges that connect the left and right part
\draw [thick] (2,1.5) -- (1,1);
\draw [thick] (2,1.5) -- (0,1);
\draw [thick] (2,.5) -- (1,1);
\draw [thick] (2,.5) -- (1,0);
\draw [thick] (2,-.5) -- (1,0);
\draw [thick] (2,-.5) -- (0,0);
% frecciolina
\coordinate [label=$\mapsto$] (B) at (-1.5,.25);
% nomi
\coordinate [label=$G$] (G) at (-3.5,-1.1);
\coordinate [label=$BL(G)$] (BLG) at (.5,-1.2);
\end{tikzpicture}
\caption{A graph and its Booth--Lueker graph. See Definition~\ref{defBL}.}
\label{fig:BL}
\end{center}
\end{figure}
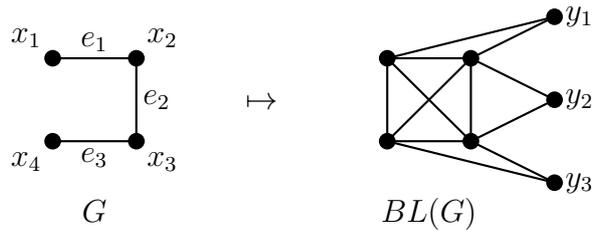

There is a bijection between squarefree ideals generated in degree~$2$ and finite simple graphs: to any graph~$G$, one associates the edge ideal $I_G$, which is generated by the monomials $x_ix_j$ for all the edges~$ij$ in the graph $G$. Fr\"oberg gave the following very famous characterization:
\begin{theorem}[Fr\"oberg, \cite{F90}]\label{frofro}
The edge ideal $I_G$ has $2$-linear resolution if and only if the complement of $G$ is chordal (i.e., every cycle of length at least four is cut by a chord, or in other words the only induced cycles in $G$ are triangles).
\end{theorem}

One can see that the complement of $BL(G)$ is chordal, and therefore the edge ideal~$I_{BL(G)}$ has a linear resolution. This and related matters are addressed in~\cite{EJO}. After that paper was made available, it was remarked that the Booth--Lueker construction could be interpreted as a map that associates to any squarefree monomial ideal generated in degree~$2$ a monomial ideal, also squarefree and generated in degree~$2$, with linear resolution. The problem of generalizing this kind of construction to any monomial ideal was raised then by Aldo Conca, and
one can see that the linearization \emph{is} such a generalization.

\medskip

We conclude  by remarking that, just as mentioned above for graphs and quadratic squarefree ideals, one can define bijections between squarefree monomial ideals of arbitrary degree and combinatorial objects (simplicial complexes or hypergraphs), and the linearity of the resolution of the ideal has a combinatorial meaning.

%\smallskip
\noindent
\emph{Simplicial complexes.} A fundamental bridge between commutative algebra and combinatorial topology is provided by the Stanley--Reisner correspondence, a bijection between squarefree monomial ideals in $\K[x_1,\dots,x_n]$ and simplicial complexes on the set~$[n]=\{1,\dots,n\}$. For the details, see for instance Chapter~1 of~\cite{MiSt}, Chapter~5 of~\cite{BH93}, or Chapter~8 of~\cite{HH}. In short, to any simplicial complex~$\Delta$ one associates the Stanley--Reisner ideal $I_\Delta$, and combinatorial properties of $\Delta$ correspond to algebraic properties of $I_\Delta$. In particular, recall that $\Lin(I)$ has a stronger property than linear resolution: it has linear quotients. The most relevant equivalence, from the point of view of this paper, is the following well-known result:
\begin{proposition}[part of Proposition~8.2.7 of~\cite{HH}]
The ideal $I_\Delta$ has linear quotients if and only if the Alexander dual $\Delta^\vee$ of $\Delta$ is shellable.
\end{proposition}
Hence, in the squarefree case, $\Lin$  can be interpreted as a map taking a pure simplicial complex and returning one, with more vertices, that has a shellable dual complex. Again, for the details we refer to Section~8.2 of~\cite{HH}.

\noindent
\emph{Hypergraphs.} Another bijection, which is less topological and more combinatorial in flavour, is between squarefree monomial ideals and hypergraphs. This is the topic of Section~\ref{hyplinres}. To each hypergraph one can associate a monomial ideal, and in~\cite{HV} the authors give a partial characterization of the hypergraphs whose associated ideal has a linear resolution. So the map $\Lin$ can also be seen as taking an arbitrary hypergraph and returning one of those ``linear hypergraphs''. See Section~\ref{hyplinres} for the details.

\subsubsection*{Acknowledgements}

I am very grateful to three people who taught me many things over the last years: Aldo Conca, who suggested in the first place the problem of defining some sort of linearization and later gave me good advice about it; my supervisor Alexander Engstr\"om, who guided me, helped me give some structure to the paper, and found counterexamples for many wrong tiny conjectures that I came up with; and Gunnar Fl\o ystad, who  very kindly helped me by posing several questions and explained to me how to think in a conceptual way about one of the main results.

\medskip

My research was funded by the Vilho, Yrjö and Kalle Väisälä Fund.

\medskip

All the experiments and computations that led to finding the results in the paper were carried out with the computer algebra system CoCoA, see~\cite{CoCoA}. For a large part of the classical results and for further reading about monomial ideals and free resolutions, we mainly refer to the excellent monographs~\cite{HH}, by Herzog and Hibi, and~\cite{Pe}, by Peeva.

%%%%%%%%%%%%%%%%%%%%%%%%%%%%%

\section{Algebraic background}\label{randomtools}

Everywhere in the paper $\K$ will be be a field and $S:=\K[x_1,\dots,x_n]$ will be the polynomial ring in $n$ variables over $\K$ equipped with the standard grading, that is, with each variable of degree~$1$. Unless otherwise stated, $\K$ will have characteristic~{zero}. By $\N$ we mean the set of non-negative integers, so that in particular $0\in\N$.

\subsection{Free resolutions and monomial ideals}

Given a graded $S$-module $M=\bigoplus_{i\in\Z}M_i$ and an integer $m\in\Z$, we denote
$$M(-m):=\bigoplus_{i\in\Z}M(-m)_i,\quad\text{where}\quad M(-m)_i:=M_{i-m},$$
the module $M$ \emph{shifted $m$ degrees}. In particular we will consider $M=S$.
%We recall that, given a finitely generated graded $S$-module $M$, as a consequence of Nakayama's lemma, every minimal system of homogeneous generators of $M$ has the same cardinality, equal to the dimension of $M/\mm M$ as a $k$-vector space. More precisely, every minimal system of homogenous generators of $M$ has the same number of elements of a fixed degree. In particular, the number
%\begin{plain}
%$$\eqalign{\maxdeg(M):=&\text{maximal degree of an element in a minimal}\cr
%&\text{system of homogeneous generators of $M$}
%}$$
%\end{plain}is an invariant of $M$.
We recall that a \tbs{free resolution} over $S$ of an $S$-module $M$ consists of a sequence~$(F_i)_{i\in\N}$ of free $S$-modules and homomorphisms $d_i\:F_i\to F_{i-1}$ such that there exists a surjective homomorphism $\vep\:F_0\to M$ and such that
$$\dots\la F_2\stackrel{d_2}\la F_1\stackrel{d_1}\la F_0\stackrel\vep\la M\la0$$
is an exact complex. The modules $M$ we will consider are finitely generated and graded.  In particular, the monomial ideals of $S$ are such modules.
A free resolution is \tbs{graded} if the maps preserve the degrees of the elements, and a it is \tbs{minimal} if $d_i(F_i)\se\mm F_{i-1}$ for each $i>0$, where $\mm=(x_1,\dots,x_n)$ is the irrelevant maximal ideal of $S$. A minimal graded free resolution is unique up to isomorphism, and sometimes we call it ``the'' minimal resolution. If we write each module of the minimal resolution as
$$F_i=\bigoplus_{j\in\Z}S(-j)^{\beta_{ij}},$$
then the natural numbers $\beta_{ij}$, also denoted by $\beta_{ij}(M)$, are invariants of $M$ called the \tbs{graded Betti numbers} of $M$.
We arrange the graded Betti numbers in the so-called \tbs{Betti table} of $M$, so that the entry in the $i$-th column and the $j$-th row is $\beta_{i,i+j}(M)$. The \tbs{projective dimension} of $M$ is the highest value of $i$ such that there is a non-zero $\beta_{i,i+j}(M)$, for some $j$. The $i$-th \tbs{total Betti number} of $M$ is $\beta_i(M):=\sum_{j\in\Z}\beta_{ij}(M)$.

\begin{definition}\label{deflinearreso}
A finitely generated graded $S$-module $M$ is said to have a \tbs{$d$-linear resolution} if $\beta_{i,j}(M)=0$ for $j\ne i+d$. That is, if all the non-zero entries of the Betti table of $M$ are in the $d$-th row.
\end{definition}

With few exceptions, the $S$-modules whose resolution we will consider in the paper will  be monomial ideals of $S$. 

\begin{notation}
Recall that for a monomial ideal $I$ there is a unique minimal system of monomial generators of $I$, consisting of the monomials in $I$ which are minimal with respect to the divisibility relation. We denote the minimal set of monomial generators by $G(I)$.
\end{notation}

Recall that a monomial $u$ belongs to a monomial ideal~$I$ if and only if $u$ is divided by some monomial in~$G(I)$.

Recall moreover that a  monomial ideal $I$ is \tbs{equigenerated} if all the elements in $G(I)$ have the same degree, and
the \tbs{support} of a monomial $u$, denoted by $\Supp(u)$, is the set of variables that divide $u$.

\subsection{Ideals with linear quotients}\label{basicsonlinearquotients}

Given two ideals $I$ and $J$ in any ring $S$, we call 
$$I:J:=\{f\in S\mid \text{$fg\in I$ for all $g\in J$}\}$$
the \tbs{colon} (or \tbs{quotient}) \tbs{ideal of $I$ with respect to $J$}. In case $J=(g)$ is a principal ideal, we denote $I:(g)=I:g$.

\begin{remark}\label{remarkcolon}
Given a monomial ideal $I\se S=\K[x_1,\dots,x_n]$ and a monomial $u\in S$, it is a well-known fact that
$$I:u=\Big(\frac f{\gcd(f,u)}\mid f\in G(I)\Big),$$
where the generators on the right-hand side might be not minimal. For a proof, see for instance Proposition~1.2.2 of~\cite{HH}.
\end{remark}

\begin{definition}
Let $I\se S=\K[x_1,\dots,x_n]$ be a homogeneous ideal. We say that $I$ has \tbs{linear quotients} if there exists a system of homogeneous generators $f_1,f_2,\dots,f_m$ of $I$ such that the colon ideal $(f_1,\dots,f_{k-1}):f_k$ is generated by linear forms for all $k\in\{2,\dots,m\}$.
\end{definition}

\begin{example}\label{linearquotexam}
\emph{The order in which we take the generators matters.} In the ring $\K[x_1,\dots,x_5]$, take  $I=(x_1x_2x_3,x_3x_4x_5,x_2x_3x_4)$. If we took the generators in the given order,  we would get in particular the colon ideal $(x_1x_2x_3):x_3x_4x_5=(x_1x_2)$,
whose only generator is quadratic. On the other hand, if we order the generators as
$$f_1=x_1x_2x_3,\quad f_2=x_2x_3x_4,\quad f_3=x_3x_4x_5,$$
then we get $(f_1):f_2=(x_1)$ and $(f_1,f_2):f_3=(x_1)$.
So indeed $I$ has linear quotients.
\end{example}

%\begin{remark}
%Let $\ell_1,\dots,\ell_r$ be $\K$-linearly independent linear forms in $S=\K[x_1,\dots,x_n]$, so that $r\le n$. Then  $\ell_1,\dots,\ell_r$ is a regular sequence. Indeed,  if we complete $\ell_1,\dots,\ell_r$ to a $\K$-basis $\ell_1,\dots,\ell_n$ of $S_1$, then $\vfi\:S\to S$ with $\vfi(x_i)=\ell_i$ for $i=1,\dots,n$ is a $\K$-automorphism, and since $x_1,\dots,x_r$ is a regular sequence it follows that $\ell_1,\dots,\ell_r$ is a regular sequence as well. INUTILE. SE TOLGO LA DIMOSTRAZIONE DELLA PROPOSIZIONE SUCCESSIVA ALLORA MEGLIO SE TOLGO ANCHE QUESTO
%\end{remark}

%The following is  a very standard result. We include a sketch of the proof, taken from~\cite{HH}, because we will refer to it later on.

\begin{proposition}[Proposition 8.2.1 of~\cite{HH}]\label{ifquotthenres}
Let $I\se S=\K[x_1,\dots,x_n]$ be a homogeneous ideal equigenerated in degree $d$ and with linear quotients. Then $I$ has a $d$-linear resolution.
\end{proposition}

%\begin{proof}[Sketch of the proof]
%Let $f_1,\dots,f_m$ be a system of homogenous generators of~$I$ with $\deg(f_i)=d$ for all $i$. We show by induction on $k$ that $I_k:=(f_1,\dots,f_k)$ has a $d$-linear resolution. For $k=1$ this is clear, so now take $k>1$. By assumption, for all $k\in\{2,\dots,m\}$ the colon ideal $J_k:=(f_1,\dots,f_{k-1}):f_k$ is generated by linear forms, so let $\ell_1,\dots,\ell_r$ be linear forms generating $J_k$ minimally. One can show that $\ell_1,\dots,\ell_r$ is a regular sequence. Therefore, the Koszul complex $K(\ell_1,\dots,\ell_r;S)$ is a minimal graded free resolution of $S/J_k$ (see Section~ A.3 of~\cite{HH}). This implies that the module
%\begin{equation}\label{torstuff}
%\Tor_i^S\big((S/J_k)(-d),\K\big)_{i+j}\iso\Tor_i^S(S/J_k,\K)_{i+(j-d)}
%\end{equation}
%vanishes for all $j\ne d$.
%We have the short exact sequence 
%$$0\la I_{k-1}\la I_k\la I_k/I_{k-1}\iso(S/J_k)(-d)\la0,$$
%which yields the long exact sequence
%$$\cdots\to\Tor_i^S(I_{k-1},\K)_{i+j}\la\Tor_i^S(I_k,\K)_{i+j}\la\Tor_i^S\big((S/J_k)(-d),\K\big)_{i+j}\to\cdots,$$
%from which we get that $\Tor_i^S(I_k,\K)_{i+j}=0$ for all $i$ and all $j\ne d$, by~\eqref{torstuff} and by induction, hence $I_k$ has linear resolution and we conclude. For additional details on the $\Tor$ functor and homological algebra in general, see for instance Appenxid~ A of~\cite{HH}.
%\end{proof}

\begin{corollary}[Corollary 8.2.2 of~\cite{HH}]\label{solitocorollariork}
Let $I\se S$ be an equigenerated homogeneous ideal with linear quotients. For each $k\in\{1,\dots,m\}$, let $r_k$ be the number of minimal generators of $(f_1,\dots,f_{k-1}):f_k$. Then
$$\beta_i(I)=\sum_{k=1}^m\binom{r_k}i.$$
In particular, $\mathrm{projdim}(I)=\max\{r_1,r_2,\dots,r_m\}$.
\end{corollary}

\begin{example}
Continuing Example~\ref{linearquotexam}, with the notation of Corollary~\ref{solitocorollariork} we get $r_1=0$, $r_2=1$ and $r_3=1$. So the projective dimension of $I$ is~$1$ and we get
$$\beta_0(I)=\sum_{k=1}^3\binom{r_k}0=1+1+1=3,\qquad\beta_1(I)=\sum_{k=1}^3\binom{r_k}1=0+1+1=2.$$
\end{example}

Next we recall a criterion that follows directly from Remark~\ref{remarkcolon}.

\begin{lemma}[Lemma 8.2.3 of~\cite{HH}]\label{amazinglemma}
A monomial ideal  $J\se\K[x_1,\dots,x_n]$ has linear quotients with respect to the monomial generators $u_1,u_2,\dots,u_t$ of $J$ if and only if
for all $j$ and $i$ with $1\le j<i\le t$ there exist $k<i$ and $\ell$ such that
$$\frac{u_k}{\gcd(u_k,u_i)}=x_\ell\quad\text{and\quad$x_\ell$ divides }\frac{u_j}{\gcd(u_j,u_i)}.$$
\end{lemma}

\subsection{Cropping the exponents from above}

\begin{notation}\label{notationcropping}
Let $I\se S= \K[x_1,\dots,x_n]$ be a monomial ideal with minimal system of generators $G(I)=\{f_1,\dots,f_s\}$. Write $f_i=x_1^{a_{1i}}\cdots x_n^{a_{ni}}$ for all $i$. Fix a vector of non-negative integers $v=(v_1,\dots,v_n)\in\N^n$. We use $v$ to ``crop'' the ideal $I$ by keeping only the generators of $I$ whose vector of exponents is componentwise at most as large as the vector $v$: that is,  we define
$$I_{\le v}:=(f_p \mid a_{pi}\le v_i\text{ for all $i=1,\dots,n$})$$
and we say that $I_{\le v}$ is obtained by \tbs{cropping $I$ by $v$}.
\end{notation}

This subsection is motivated by the following well-known result: if $I\se \K[x_1,\dots,x_n]$ is a monomial ideal which has linear resolution, then $I_{\le v}$ still has linear resolution for any $v\in\N^n$ (see Section~56 of~\cite{Pe} for a general treatment). In what follows we  prove an analogous result for linear quotients, Proposition~\ref{cropparequozientilinear}. Before we state it, observe that with the notation above, by Remark~\ref{remarkcolon}, we have
$$(f_1,\dots,f_{i-1}):f_i=\Big(\frac{f_k}{\gcd(f_k,f_i)}\mid k\in\{1,\dots,i-1\}\Big),$$
where more explicitly we may write
\begin{align*}
\frac{f_k}{\gcd(f_k,f_i)}&=\frac{x_1^{a_{1k}}\cdots x_n^{a_{nk}}}
{x_1^{\min\{a_{1k},a_{1i}\}}\cdots x_n^{\min\{a_{nk},a_{ni}\}}}\\
&=x_1^{a_{1k}-\min\{a_{1k},a_{1i}\}}\cdots x_n^{a_{nk}-\min\{a_{nk},a_{ni}\}}.
\end{align*}

\begin{proposition}\label{cropparequozientilinear}
Let $I=(f_1,\dots,f_s)\subset\K[x_1,\dots,x_n]$ be a monomial ideal with linear quotients with respect to the given ordering of the generators. Fix $v\in\N^n$ and denote $I_{\le v}=(f_{b_1},\dots,f_{b_t})$, where $b_1<\dots<b_t$ are the indexes of the generators that survided the cropping. Then $I_{\le v}$ has linear quotients with respect to $f_{b_1},\dots,f_{b_t}$.
\end{proposition}

\begin{proof}
By Lemma~\ref{amazinglemma}, $I_{\le v}$ has linear quotients with respect to $f_{b_1},\dots,f_{b_t}$ if and only if, for all $1\le j<i\le t$, there exist $k<i$ and $\ell$ such that
\begin{equation}\label{condlemmaquozienti}
\frac{f_{b_k}}{\gcd(f_{b_k},f_{b_i})}=x_\ell\quad\text{and\quad$x_\ell$ divides }\frac{f_{b_j}}{\gcd(f_{b_j},f_{b_i})}.
\end{equation}
We will prove~\eqref{condlemmaquozienti} starting from the analogous condition for $I$. Fix $j<i$, or equivalently $b_j<b_i$. Because $I$ has linear quotients with respect to the given order of the generatos, there exist $p<b_i$ and  $\ell$ such that
\begin{equation}\label{dalleipotesisuI}
\frac{f_p}{\gcd(f_p,f_{b_i})}=x_\ell\quad\text{and\quad$x_\ell$ divides }\frac{f_{b_j}}{\gcd(f_{b_j},f_{b_i})}.
\end{equation}
This would prove~\eqref{condlemmaquozienti} if we could show that $p=b_k$ for some $k$, or in other words we want to show that the vector of exponents of $f_p=x_1^{a_{1p}}\cdots x_n^{a_{np}}$ is below the cropping vector $v=(v_1,\dots,v_n)$. 
The first of the two things in~\eqref{dalleipotesisuI} more explicitly means that
$$\begin{cases}
a_{jp}=\min\{a_{jp},a_{jb_i}\}&\text{for $j\in[n]\setminus\{\ell\}$,}\\
a_{\ell p}=\min\{a_{\ell p},a_{\ell b_i}\}+1,
\end{cases}$$ that is,
$$\begin{cases}
a_{jp}\le a_{jb_i}\ (\le v_j)&\text{for $j\in[n]\setminus\{\ell\}$,}\\
a_{\ell p}=a_{\ell b_i}+1.
\end{cases}$$
The second part of~\eqref{dalleipotesisuI} is more explicitly $a_{\ell b_j}-\min\{a_{\ell b_j},a_{\ell b_i}\}\ge1$, or equivalently $a_{\ell b_j}\ge a_{\ell b_i}+1$. Putting these things together with $a_{\ell b_j}\le v_\ell$, which we have by construction, we get $a_{\ell p}\le v_\ell$, so that $f_p$ actually is among the generators of $I_{\le v}$.
\end{proof}

The assumption about the vector is important. One cannot simply kill some random generators. Consider for instance the following example: the ideal $(xy^2,xyz,xz^2)$ has linear quotients, but the ideal $(xy^2,xz^2)$ doesn't. And indeed we are not cropping by any vector to get the latter ideal.

%\begin{remark}
%Here come some final remarks on this topic. Write $f_i=x_1^{a_{1i}}\dots x_n^{a_{ni}}$ for all $i$. Then we can arrange the exponents in a matrix:
%$$A=\left(\begin{matrix}
%a_{11}&\cdots& a_{1s}\\
%\vdots&&\vdots\\
%a_{n1}&\cdots&a_{ns}
%\end{matrix}\right)\in M^{n\times s}(\N).$$
%Is there a nice way of encoding the linear quotients story by means of matrices? Maybe in a tropical way (since a minimun comes up, but also a minus)? Or actually I can ask myself what it means combinatorially (i.e., for the exponents) that $I$ has linear quotients. I think it means that if we have
%$$a_{{p_1}k}-\min\{a_{{p_1}k},a_{{p_1}i}\}>0,\dots,a_{{p_t}k}-\min\{a_{{p_t}k},a_{{p_t}i}\}>0,\quad t>1,$$
%then there exist $h\in\{1,\dots,t\}$ and $q\ne p$ such that
%$$\begin{cases}
%a_{{p_h}k}-\min\{a_{{p_h}k},a_{{p_h}i}\}>1,\\
%a_{qk}-\min\{a_{qk},a_{qi}\}=0&\text{for $q\in[n]\setminus\{p_h\}$.}
%\end{cases}$$
%Is there a way to encode this by means of matrices, or does it look like something already known in combinatorics?
%\end{remark}

%%%%%%%%%%%%%%%%%%%%%%%%%%%%%%%%%
%%%%%%%%%%%%%%%%%%%%%%%%%%%%%%%%%

\section{Linearization of  equigenerated monomial ideals}\label{lineargeneral}

Let $\K$ be a field of characteristic zero and let  $S:=\K[x_1,\dots,x_n]$ be the polynomial ring in $n$ variables over $\K$.

\begin{notation}\label{notazlin}
Let $I\se S$ be a monomial ideal  with minimal system of monomial generators $G(I)=\{f_1,\dots,f_m\}$ and let $I$ be \tbs{equigenerated}, namely such that all the generators have the same degree $d$. For each $i\in\{1,\dots,n\}$, let $M_i$ be the highest exponent with which $x_i$ occurs in $G(I)$.
\end{notation}

\begin{definition}\label{deflinearization}
The \tbs{linearization of $I$}, inside the polynomial ring $R:=\K[x_1,\dots,x_n,y_1,\dots,y_m]$, is the ideal 
\begin{align*}
\Lin(I)&:=\big(x_1^{a_1}\cdots x_n^{a_n}\mid a_1+\dots+a_n=d\,\text{ and $a_i\le M_i$ for all $i$}\big)\\
&\ \quad+\big(f_jy_j/x_k \mid \text{$x_k$ divides $f_j$},\ k=1,\dots,n,\ j=1,\dots,m\big).
\end{align*}
We refer to the two summands respectively as the \tbs{complete part} and the \tbs{last part} of $\Lin(I)$. 
\end{definition}

The name ``complete part'' comes from generalizing the Booth--Lueker construction, see Definition~\ref{defBL}. There, the analogous set of generators corresponds to a complete graph. Notice moreover that we can also write the complete part as ${(x_1,\dots,x_n)^d}_{\le(M_1,\dots,M_n,0,\dots,0)}$, see Notation~\ref{notationcropping}. The name ``last part'' is due to the lack of a better name.

\begin{remark}
Observe that the given generators of $I$ are minimal and $\Lin(I)$ is equigenerated in the same degree $d$ as $I$.
\end{remark}

\begin{remark}[algorithm to retrieve $I$ from $\Lin(I)$]\label{tornarindietrosingologrado}
For all $j\in\{1,\dots,m\}$, let $f_{j,1},f_{j,2},\dots,f_{j,r_j}$ be the monomials such that 
\begin{equation}\label{retreivinglcm}
f_{j,1}y_j,\ f_{j,2}y_j,\dots,\ f_{j,r_j}y_j
\end{equation}
are the minimal generators of $\Lin(I)$ divisible by the variable $y_j$. For each entry equal to $d$ in the vector $(M_1,\dots,M_n)$, we have a generator $f_j$ of $I$ which is the $d$-th power of the corrsponding $x$-variable. These generators $f_j$ are such that $r_j=1$. Otherwise, if $j$ is such that $r_j>1$, then we may retrieve the corresponding generator of $I$ as
$$f_j=\lcm(f_{j,1},f_{j,2},\dots,f_{j,i_j})\qquad\text{for all $j\in\{1,\dots,m\}$}.$$
%Actually, in most cases it's enough to know the generators of the last part of $\Lin(I)$ in order to recover $I$, but it's important to know the complete part of $\Lin(I)$ to detect generators of $I$ which are powers of variables.
%Now, a problem might occur if we have a generator of $I$ which is a power of a  single variable, say for instance $f_j=x_1^4$. Then the only generator of $\Lin(I)$ involving $y_j$ would be $x_1^3y_j$, and taking the least common multiple as above would not work. A special case of the problem just mentioned can occur if we have a generator which is just a single variable, say $f_j=x_1$. Then in the last part of $\Lin(I)$ we get the generator $y_j$, with no $x$-variables in its support. This is why it's \emph{important} to know the complete part: in case we have degree~$d=1$, it's enought to know the vector $(M_1,\dots,M_n)$ in order to deduce which generators (variable corresponding to $M_i=1$) we have. So now let's assume we have arbitrary degree~$d$ and a generator which is a power of a variable, say $f_j=x_1^d$. If another generator uses the variable $x_1$, then of course it will be contained with exponent smaller than $d$. 
\end{remark}

\begin{definition}\label{defstarlinearization}
The \tbs{$*$-linearization of $I$}, in the polynomial ring $R:=\K[x_1,\dots,x_n,y_1,\dots,y_m]$, is the ideal 
\begin{align*}
\LIN(I)&:=\big(x_1^{a_1}\cdots x_n^{a_n}\mid a_1+\dots+a_n=d\,\text{ and $a_i\le M$ for all $i$}\big)\\
&\ \quad+\big(f_jy_j/x_k \mid \text{$x_k$ divides $f_j$},\ k=1,\dots,n,\ j=1,\dots,m\big).
\end{align*}
We refer to the two summands respectively as the \tbs{complete part} and the \tbs{last part} of $\LIN(I)$. 
\end{definition}

\begin{remark}
The last parts of $\Lin(I)$ and $\LIN(I)$ are always equal, the only difference is in the complete part.
Notice that, since $M_i\le M$ for each $i\in\{1,\dots,n\}$, we always have 
$$\Lin(I)\se\LIN(I),$$
with equality only if $M_1=M_2=\dots=M_n=M$.   The reason for the introduction of $\LIN(I)$, which is ``coarser'' than $\Lin(I)$, is that given its symmetry it's easier to understand than $\Lin(I)$ in the non-squarefree case. In the squarefree case, treated in Section~\ref{linearsquarefree} the difference between $\Lin(I)$ and $\LIN(I)$ is not so significant, see Remark~\ref{differenzaLinLIN}. And moreover in that case $\LIN(I)$ is a more direct generalization of the Booth--Lueker construction in Definition~\ref{defBL}.
\end{remark}

%\begin{proposition}[number of generators]
%IT SEEMS THAT the number of generators of the complete part and the last part of $\LIN(I)$ are, respectively, the summands in 
%$$\sum_{i=0}^{\lfloor\frac d{M+1}\rfloor}(-1)^i\binom ni\binom{d+n-1-i(M+1)}{n-1}+\sum_{j=1}^m\#\Supp\{f_j\}.$$
%\end{proposition}
%
%It's clear that the second sum specializes to $md$ when $I$ is squarefree, but it's a bit less clear that one has
%$$\sum_{i=0}^{\lfloor d/2\rfloor}(-1)^i\binom ni\binom{d+n-1-2i}{n-1}=\binom nd,$$
%but \cocoa\ confirms this equality in all the examples, so\dots MAYBE WE CAN GET RID OF THIS LITTLE THING, RIGHT?

\subsection{Properties of the linearization}

\emph{Lexicographic order.} Let $u=x_1^{a_1}\cdots x_n^{a_n}$ and $v=x_1^{b_1}\cdots x_n^{b_n}$ be two monomials in $\K[x_1,\dots,x_n]$, with $a=(a_1,\dots,a_n)$ and $b=(b_1,\dots,b_n)$ in~$\N^n$ their respective vectors of exponents. We recall that $u$ is larger than $v$ in lexicographic order, written $u>_{\Lex}v$, if the leftmost non-zero entry of $a-b$ is positive. We equivalently write $a>_{\Lex}b$ to mean the same thing.

The following is the main result of the paper.

\begin{theorem}\label{mainthm}
Assume that $f_1,\dots,f_m$ are in decreasing lexicographic order.
List the generators of the complete part of $\Lin(I)$, and respectively of $\LIN(I)$, in decreasing lexicographic order.  List the generators $\frac{f_j}{x_k}y_j$ of the last part  first by increasing\til$j$, and secondly by increasing~$k$.
The ideals $\Lin(I)$ and $\LIN(I)$ have linear quotients with respect to the given ordering of the generators.
\end{theorem}

\begin{proof}
Denote $P:=(x_1,\dots,x_n)^d$ and let $L$ be the last part of $\Lin(I)$, which is the same as the last part of $\LIN(I)$.
It's enough to prove that $P+L$ has linear quotients, because then, in order to conclude that $\Lin(I)$ and $\LIN(I)$ have linear quotients, we apply Proposition~\ref{cropparequozientilinear} cropping from above by the vector $(M_1,\dots,M_n,1,\dots,1)$ and $(M,\dots,M,1,\dots,1)$, respectively. 

With the notation of Lemma~\ref{amazinglemma}, we have three cases: (1)~ we pick both $u_j$ and $u_i$ in $G(P)$, (2)~ we pick $u_j$ in $G(P)$ and $u_i$ in $G(L)$, or (3)~ we pick both $u_j$ and $u_i$ in $G(L)$. Of course in some special situations (like for instance if $I$ is principal) we cannot do this, but in those cases the claim is trivial.
\begin{enumerate}
\item[(1)] Let $u_j=x_1^{a_1}\cdots x_n^{a_n}$ and $u_i=x_1^{b_1}\cdots x_n^{b_n}$ be in $G(P)$, so that $\sum_\iota a_\iota=\sum_\iota b_\iota=d$, and $a>_\Lex b$. Let $\ell$ be the index of the leftmost non-zero entry of $a-b$, so that $a_\ell-b_\ell>0$. Notice that $\ell<n$, because otherwise we could not have $\sum_\iota a_\iota=\sum_\iota b_\iota$. Then there is some $q>\ell$ such that $b_q>a_q\ge0$, which means that $x_q$ divides $u_i$. 
If we pick $u_k:=\frac{x_\ell}{x_q}u_i$, then we get
$$\frac{u_k}{\gcd(u_k,u_i)}=x_\ell,$$
and  by construction of~$\ell$ also the second property in Lemma~\ref{amazinglemma} holds, namely $x_\ell$ divides $u_j/\gcd(u_j,u_i)$.
\item[(2)] Let $u_j=x_1^{a_1}\cdots x_n^{a_n}\in G(P)$ and $u_i=\frac{f_ty_t}{x_r}\in G(L)$, where $f_t=x_1^{b_1}\cdots x_n^{b_n}$ is some generator of $I$ and $x_r$ is a variable dividing $f_t$. We may write explicitly $u_i=x_1^{b_1}\cdots x_r^{b_r-1}\cdots x_n^{b_n}y_t$. Denote
$$X:=\Supp\Big(\frac{u_j}{\gcd(u_j,u_i)}\Big)=\Big\{x_p\mid a_p>
\begin{cases}
b_p&\text{for $p\ne r$}\\
b_r-1&\text{for $p=r$}
\end{cases}\Big\}.$$
We have $X\ne\emse$, because $\sum_\iota a_\iota=\sum_\iota b_\iota$. The index $r$ might be in~$X$ or not. We distinguish two cases:
\begin{itemize}
\item $X=\{r\}$: We have $a_r>b_r$, and we may pick $u_k=f_t$ to conclude.
\item $X\ne\{r\}$: Pick $\ell\in X\setminus\{r\}$ and
$$u_k:=x_\ell^{b_\ell+1}x_r^{b_r-1}\prod_{p\notin\{\ell,r\}}x_p^{b_p}=\frac{x_\ell}{x_r}f_t.$$
\end{itemize}
\item[(3)] Let both $u_j$ and $u_i$ be in $G(L)$. This means that we have
$$u_j=\frac{f_\mu}{x_\alpha}y_\mu,\quad u_i=\frac{f_\lambda}{x_\beta}y_\lambda,\quad\text{with $f_\mu,f_\lambda\in G(I)$ and $x_\alpha|f_\mu$, $x_\beta|f_\lambda$.}$$
Denote $f_\mu=x_1^{a_1}\cdots x_n^{a_n}$ and $f_\lam=x_1^{b_1}\cdots x_n^{b_n}$. We consider two different cases, the second of which has again some subcases.
\begin{itemize}
\item $\mu= \lambda$: Pick $u_k:=u_j$.
\item $\mu\ne \lambda$: We have $\mu<\lam$ by assumption, so that $f_\mu>_\Lex f_\lam$. Let $\ell$ be the index of the leftmost non-zero entry of $a-b$. There are again two possibilities.
\begin{itemize}
\item $\alpha=\beta$: Pick $u_k:=\frac{x_\ell}{x_\alpha}f_\lambda$, which works also in case $\ell=\alpha$.
\item $\alpha\ne\beta$: There are still two subcases.
\begin{itemize}
\item $\ell\ne\alpha$: Pick $u_k:=\frac{x_\ell}{x_\beta}f_\lam$, which works also in case $\ell=\beta(\ne\alpha)$.
\item $\ell=\alpha$: We have three final sub-subcases. In the first we have
$$\begin{cases}
a_\alpha=b_\alpha+1\\
a_\beta=b_\beta-1\\
a_p=b_p&\text{for $p\ne \alpha,\beta$},
\end{cases}$$
(so that $\alpha<\beta$). In this case pick $u_k:=\frac{f_\mu}{x_\alpha}y_\mu=\frac{f_\lam}{x_\beta}y_\mu$. In the case where $a_\alpha>b_\alpha+1$, pick $u_k:=\frac{x_\alpha}{x_\beta}f_\lam$. In the case where $a_\alpha=b_\alpha+1$ and $a_\gamma>b_\gamma$ for some $\gamma\notin\{\alpha,\beta\}$, pick $u_k:=\frac{x_\gamma}{x_\beta}f_\lam$.
\end{itemize}
\end{itemize}
\end{itemize}
\end{enumerate}
All possible cases are included, and the proof is then complete.
\end{proof}

As an immediate consequence, by Proposition~\ref{ifquotthenres}, we get the following.

\begin{corollary}\label{Linlinres}
The ideals $\Lin(I)$ and $\LIN(I)$ have $d$-linear resolution.
\end{corollary}

\subsubsection{Functorial properties}\label{sectionwithfunctor}

The ideals of a fixed ring $A$ are the objects of a category $\Ideals(A)$, with morphisms consisting of the inclusions. In particular, in this category two objects being isomorphic just means that they are equal as sets. For a polynomial ring~$S$ we have the subcategory $\MonIdeals(S)$, where the objects are monomial ideals and the morphisms are again the inclusions. 
Naively speaking, we would like
$$\Lin\:\MonIdeals(S)\la\MonIdeals(R)$$  to be a functor, which is just a fancy wat to say that if $I\se J$ (equivalently, $G(I)\se G(J)$) then $\Lin(I)\se \Lin(J)$. There is no problem for what concerns the complete part:  the vector of exponents can get larger, if we enlarge the set of generators, hence inclusions are preserved.

On the other hand the last part causes trouble: the last part, and in fact the ring $R$ itself, depends on the generators of $I$. The variable $y_j$ is associated to the generator $f_j$, and this does not behave well if in the inclusion $I\se J$ some generators are listed in a different way. So, we might change the notation in Definition\til\ref{deflinearization} by indexing the $y$-variables $y$  on the generators of $I$ and not on the position of these generators in the list. 
%\begin{example}
%Consider $I=(x_1x_2x_4)$ and $J=(x_1x_2x_3,x_1x_2x_4)$ inside the ring $\K[x_1,\dots,x_4]$. Then clearly $I\se J$. The issue, if we act according to Definition~\ref{deflinearization}, would be that the only generator for $I$ would give us a variable $y_1$, and the generators of $J$ would give variables $y_1$ and $y_2$. Notice that the one giving variable $y_2$ corresponds to the only generator of $I$, which gives variable $y_1$. As we do not want this to happen, we call $u=x_1x_2x_4$ and $v=x_1x_2x_3$, so that the generators of $I$ and $J$ give variables $y_u$ for $\Lin(I)$ and $y_v$ and $y_u$ for $\Lin(J)$.
%\end{example}
%Namely, we would then consider something like
%$$R=\K[x_1,\dots,x_n,y_f\mid f\in G(I)].$$
The next problem is then the amount of $y$-variables in the ring~$R$: we can either index them on all possible monomials in $x_1,\dots,x_n$ and have a functor~$\Lin$ that can take as input all ideals of $S$, or we can have a functor
$$\Lin_d\:\MonIdeals(S)_d\la\MonIdeals(R)_d,$$
for each degree $d$, as follows.
\begin{fundefinition}\label{funlin}
Referring to Notation~\ref{notazlin}, the \tbs{linearization of~$I$}, in the polynomial ring
$$R:=\K[x_1,\dots,x_n,y_u\mid\text{$u$ monomials in $x_1,\dots,x_n$ of degree $d$}],$$
 is the ideal 
\begin{align*}
\Lin_d(I)&:=\big(x_1^{a_1}\cdots x_n^{a_n}\mid a_1+\dots+a_n=d\,\text{ and $a_i\le M_i$ for all $i$}\big)\\
&\ \quad+\Big(\frac{f_j}{x_k}y_{f_j} \mid \text{$x_k$ divides $f_j$},\ k=1,\dots,n,\ j=1,\dots,m\Big).
\end{align*}
\end{fundefinition}
We stress again that this small change makes no real difference, it's just a small technicality.
When we  generalize the linearization construction to arbitrary monomial ideals in Section~\ref{linnonequig}, the $y$-variables of $R$ may be indexed on all monomials \emph{up to} degree $d$, so that we get  functors
$\Lin_{\le d}\:\MonIdeals(S)_{\le d}\to\MonIdeals(R)_{\le d}$.

\begin{remark}
To sum up, if we change the notation in the definition of linearization as suggested in the discussion above, we get that $\Lin(I)\se \Lin(J)$ if $I\se J$. Same thing goes for $\LIN(I)\se\LIN(J)$ if $I\se J$.
\end{remark}

\subsubsection{Interplay of linearization and standard operations}\label{linandops}

In the cateogory $\MonIdeals(S)$, starting from two ideals $I$ and $J$ one can construct for instance $I+J$, $IJ$ and $I\cap J$, and these are all still monomial ideals and well understood. Unfortunately, it seems that $\Lin$ is really well-behaved only with respect to taking the sum.

\begin{proposition}
If we use the Functorial Definition~\ref{funlin}, then the functors $\Lin$ and $\LIN$ are such that 
$$\Lin(I)+\Lin(J)\se\Lin(I+J)\quad\text{and}\quad \LIN(I)+\LIN(J)\se\LIN(I+J),$$
where $I$ and $J$ are equigenerated in the same degree.
\end{proposition}

\begin{proof}
We can decompose $\Lin(I)=C_I+L_I$ and similarly $\Lin(J)=C_J+L_J$ in complete and last part. Denote by $v_I$ the vector of highest exponents of $I$ and similarly $v_J$ and $v_{I+J}$ for $J$ and $I+J$. We have $M_i(I+J)=\max\{M_i(I),M_i(J)\}$, so that $v_I\le v_{I+J}\ge v_J$ and this means that $C_I+C_J\se C_{I+J}$. As for the last part, we have $L_I+L_J=L_{I+J}$. So this settles the first inclusion in the statement. For what concerns the second inclusion, the only difference is in the complete part of the $*$-linearization, which is ``coarser''. In this case we simply have $M_{I+J}=\max\{M_I,M_J\}$, and this is enough.
\end{proof}

With the notation as above, the equality $ \LIN(I)+\LIN(J)=\LIN(I+J)$ holds if and only if $M_I=M_J$, and
the equality $\Lin(I)+\Lin(J)=\Lin(I+J)$ holds if and only if $v_I=v_J$.

\begin{remark}[linearization and products]\label{productnogood}
The behaviour of $\Lin$ (and similarly, $\LIN$) with respect to taking the product is a bit more complicated.  If $I$ and $J$ are equigenerated respectively in degrees $d_I$ and $d_J$, then both  $\Lin(I)\Lin(J)$ and $\Lin(IJ)$ are equigenerated in the same degree $d_Id_J$. Notice that there cannot be an inclusion $\Lin(I)\Lin(J)\se\Lin(IJ)$: each generator in the last part of $\Lin(IJ)$ contains only one $y$-variable with exponent~$1$, but some of the minimal generators are quadratic in the $y$-variables.
One might then ask if there is any hope of having
$$\Lin(IJ)\subset\Lin(I)\Lin(J).$$
Unfortunately the ``improved'' version  of linearization in Functorial Definition~\ref{funlin} cannot possibly work: the ring in which $\Lin(IJ)$ is defined has $y$-variables indexed on monomials of degree $d_Id_J$, whereas the one in which $\Lin(I)\Lin(J)$ lives has $y$-variables indexed on both monomials of degree $d_I$ and $d_J$.
 The problem is not in the complete part, but rather in the last part of $\Lin(IJ)$. If $G(I)=\{f_1,\dots,f_m\}$ and $G(J)=\{g_1,\dots,g_m\}$, in $\Lin(IJ)$, then in the last part of $\Lin(IJ)$ we have generators
$$\frac{f_ig_j}{x_k}y_\ell$$
where $x_k$ divides $f_ig_j$. Hence, $x_k$ divides at least one of $f_i$ or $g_j$, say $f_i$. So then we want to have an indexing of the $y$-variables such that $\frac{f_i}{x_k}y_\ell$ actually is in the last part of $\Lin(I)$. If we have that, then of course $g_j$ is in the complete part of $\Lin(J)$ and then we have $\frac{f_ig_j}{x_k}$ in $\Lin(I)\Lin(J)$. There might be a way to overcome the problem, if we one comes up with a different way of indexing the variables. Or perhaps  relaxing the notion of morphism in the category, making it  ``looser'' than set-theoretic inclusion, and allowing some change of variables.
\end{remark}

\begin{remark}[intersections]
With the intersection we have similar problems as in the case of products. It's not clear how to define the linearization in a sensible way that allows to compare for instance $\Lin(I\cap J)$ and $\Lin(I)\cap\Lin(J)$.
Besides, there is an additional problem: notice that even if $I$ and $J$ are equigenerated in the same degree, $I\cap J$ might not be. Consider for instance 
$$I=(x_1^2x_2,x_2x_3x_4)\quad\text{and}\quad J=(x_1x_2)$$
in $\K[x_1,x_2,x_3]$. Then $I\cap J=(x_1^2x_2,x_1x_2x_3x_4)$ is not equigenerated. This problem might be overcome by applying $\Lin$ or $\LIN$ to the \emph{equification} of $I\cap J$, introduced in Definition~\ref{defequification}. This in fact is the way in which we define the linearization of an arbitrary, not necessarily equigenerated monomial ideal, in Definition~\ref{linpertutti}. In this example we have
$$(I\cap J)^\eq=(x_1^2x_2z,x_1x_2x_3x_4)\quad \subset\quad\K[x_1,x_2,x_3,z],$$
which looks promising,
but one still needs to introduce an appropriate way to index the $y$-variables.
\end{remark}

%%%%%%%%%%%%%%%%%%%%%%%%%%%%

\subsubsection{Polymatroidal ideals and linearization}

The study of matroids and polymatroids is a core area in combinatorics and related fields. Polymatroidal ideals, which can be defined even without referring explicitly to any polymatroid, constitute an interesting large class of equigenerated ideals which are particularly well-behaved with respect to resolutions. For instance, all powers of a polymatroidal ideal have a linear resolution (see Section~5 of~\cite{CoHe} and Section~5 of~\cite{BrCo}). In particular, Herzog and Takayama proved in~\cite{HeTa}  that polymatroidal ideals have linear quotients. So it's natural to ask when $\Lin(I)$ is polymatroidal. It turns out that in most cases $\Lin(I)$ is \emph{not} polymatroidal. For additional information on polymatroidal ideals, see Section~12.6 of~\cite{HH}.

In the following, for a monomial $u=z_1^{a_1}\cdots z_s^{a_s}$ in the variables $z_1,\dots,z_s$, we write $\deg_{z_i}(u):=a_i$.

\begin{definition}
An equigenerated monomial ideal $J\se\K[z_1,\dots,z_s]$  is \tbs{polymatroidal} the following holds: for any $u$ and $v$ in $G(J)$ and for any\til$i$ such that $\deg_{z_i}(u)>\deg_{z_i}(v)$, there exists~$j$ such that $\deg_{z_j}(u)<\deg_{z_j}(v)$ and $\frac u{z_i}z_j\in G(J)$.
\end{definition}

\begin{theorem}\label{carattpolymatr}
For a monomial ideal $I\subset S=\K[x_1,\dots,x_n]$ equigenerated in degree~$d$, in the following cases $\Lin(I)$ is polymatroidal:
\begin{itemize}
\item $d=1$, that is, $I$ is generated by some variables; 
\item $d$ is arbitrary and $I$ is principal.
\end{itemize}
In all other cases $\Lin(I)$ is not polymatroidal.
\end{theorem}

\begin{proof}
Assume that $I$ is generated by some variables of $S$.  Then $\Lin(I)$ is also generated by variables, and therefore trivially polymatroidal. Assume now that $d$ is arbitrary and $I=(f)$ is principal. Denoting simply by~$y$ the only $y$-variable, we have
$$\Lin(I)=(f)+\Big(\frac f{x_k}y\mid x_k\in\Supp(f)\Big).$$
It's immediate to see that, picking any two generators, the condition in the definition of polymatroidal ideal is satisfied.

Next we prove that in all the other cases, namely if $d\ge2$ and $I$ is not principal, $\Lin(I)$ is not polymatroidal. Let 
$$f=x_1^{c_1}\cdots x_n^{c_n}\quad\text{and}\quad g=x_1^{e_1}\cdots x_n^{e_n}$$
be two distinct minimal generators of $I$, so that $\sum_\iota d_\iota=\sum_\iota e_\iota=d$. In what follows we use the notation introduced in the Functorial Definition~\ref{funlin}. We distinguish two cases, the second of which has several subcases:
\begin{itemize}
\item[(1)] If there exists  an index~$k$ such that $c_k\ge e_k+2$, then we pick $u:=\frac f{x_k}y_f$ and $v:=\frac g{x_q}y_g$ for some $q\ne k$. We  select $x_k$ to be the variable $z_i$ in the definition of polymatroidal, and observe that $\deg_{x_k}(u)>\deg_{x_k}(v)$. But there exists no variable~$x_j$ with $\deg_{x_j}(u)<\deg_{x_j}(v)$ and  such that the monomial
$$\frac u{x_k}x_j=\frac f{x_k^2}x_jy_f$$
is in $G(\Lin(I))$. Indeed, the only way this monomial could be in $G(\Lin(I))$ would be if $\frac f{x_k^2}x_j=\frac f{x_\ell}$ for some $\ell$, but this is possible only if $x_j=x_k$, and this cannot happen because of the above assumptions on $x_k$ and $x_j$. Therefore  in this case $\Lin(I)$ is not polymatroidal.
\item[(2)] So now suppose that there is no  $k$ such that $c_k\ge e_k+2$. And swapping~$u$ and~$v$ we may actually assume that
$$c_i\le e_i+1\quad\text{and}\quad e_i\le c_i+1$$
for all $i$. There must be at least one index $k$ such that $c_k=e_k+1$. We distinguish again two subcases:
\begin{itemize}
\item If $e_k=1$, then $c_k=2$. If we pick
$$u:=\frac f{x_k}y_f\quad\text{and}\quad v:=\frac g{x_k}y_g,$$ 
then $1=\deg_{x_k}(u)>\deg_{x_k}(v)=0$, and of course there is no~$x_j$ with $\deg_{x_j}(v)>\deg_{x_j}(u)$ and such that $\frac u{x_k}x_j=\frac f{x_k^2}x_jy_f\in G(\Lin(I))$.
\item Suppose that there is no $k$ such that $e_k=1$ and $c_k=2$, and also no $k$ such that $e_k=2$ and $c_k=1$. %Hence we have
%$$c_i=e_i\qquad\text{for all $c_i\notin\{0,1\}$ and $e_i\notin\{0,1\}$}.$$
Since $f\ne g$, we know that that there are~$k$ and $p$ such that $c_k=e_k+1=1$ and $c_p=e_p-1=0$. We divide in two cases one last time:
\begin{itemize}
\item If there is $\ell\ne k$ such that $\deg_{x_\ell}(f)>\deg_{x_\ell}(g)$, meaning that~$c_\ell=1$ and~$e_\ell=0$, then %there also exists $q\ne p$ such that $\deg_{x_q}(f)<\deg_{x_q}(g)$, and 
we can pick
$$u:=\frac f{x_k}y_f\quad\text{and}\quad v:=\frac g{x_p}y_g$$
and then consider $u/x_\ell$. There is no $x_j$ such that $\deg_{x_j}(u)<\deg_{x_j}(v)$ and $\frac u{x_\ell}x_j=\frac f{x_kx_\ell}x_j\in G(\Lin(I))$.
\item Assume then that $c_i=e_i$ for all $i\notin\{k,p\}$. Then, since~$d\ge2$, there exists some $\ell\notin\{k,p\}$ such that $\deg_{x_\ell}(f)=\deg_{x_\ell}(g)>0$, so we can choose
$$u:=\frac f{x_\ell}y_f\quad\text{and}\quad v:=\frac g{x_\ell}y_g,$$
and consider $u/x_k$.
\end{itemize}
\end{itemize}
\end{itemize}
In all the possible cases we produced two generators that show how the condition in the definition of polymatroidal ideal is not satisfied.
\end{proof}

%%%%%%%%%%%%%%%%%%%%%%%%%%

\subsection{Radical of $\Lin(I)$ and $\LIN(I)$}

Recall that, given an ideal $I\se S=\K[x_1,\dots,x_n]$, the \tbs{radical of $I$} is defined as
$$\sqrt I:=\{f\in S\mid f^p\in I\text{ for some $p\in\N$}\}.$$
In case $I$ is a monomial ideal, $\sqrt I$ is also a monomial ideal and one has an easy way to compute generators for $\sqrt I$, described as follows. We use the same notation as in~\cite{HH} and, for a monomial $u=x_1^{a_1}x_2^{a_2}\cdots x_n^{a_n}$, we denote
$$\sqrt u:=\prod_{\substack{i=1,\dots,n\\a_i>0}}x_i.$$
Then, if $I$ is a monomial ideal generated by monomials $u_1,\dots,u_s$, the radical of $I$ is simply $\sqrt I=(\sqrt{u_i}\mid i=1,\dots,s)$. See for instance~\cite{HH}, Proposition~1.2.4. We say that an ideal $I$ is a \tbs{radical ideal} if $I=\sqrt I$.

A first question that might arise is: when is $\Lin(I)$ (or $\LIN(I)$) a radical ideal? For a monomial ideal, being radical is equivalent to being squarefree, that is, having all minimal monomial generators which are squarefree. Moreover, as explained in Remark~\ref{remarksqlin}, $\Lin(I)$ is squarefree, or equivalently $\LIN(I)$, if and only if $I$ is squarefree. (For more details on this situation see Section~\ref{linearsquarefree}.) For this reason, we assume $M\ge2$ in this subsection.

\begin{remark}
It is trivial to notice that applying $\Lin$ or $\LIN$ to the radical $\sqrt I$ of some equigenerated ideal $I$ might not make sense, because $\sqrt I$ might not be equigenerated. Consider for instance  $I=(x_1^2,x_2x_3)$, which is such that $\sqrt I=(x_1,x_2x_3)$. (A possibility could be applying the radical to the equification $I^\eq$, see Definitiln~\ref{defequification}, but we did not manage to get anything fruitful from that approach.) Our attention will then be focused on applying to $I$ first the linearization and then the radical.
It turns out that $\sqrt{\LIN(I)}$ is easier to understand and seems to have nicer properties in general than $\sqrt{\Lin(I)}$. The following example illustrates this.
\end{remark}

\begin{example}
Consider $I=(x_1^2x_2,x_1x_2x_3)\subset \K[x_1,x_2,x_3]$. Then one has
$$\Lin(I)=(x_1^2x_2,\ x_1^2x_3,\ x_1x_2x_3,\ x_1x_2y_1,\ x_1^2y_1,\ x_2x_3y_2,\ x_1x_3y_2,\ x_1x_2y_2)$$
and we get
$$\sqrt{\Lin(I)}=(x_1x_2,\ x_1x_3,\ x_1y_1,\ x_2x_3y_2),$$
which is somewhat  difficult to describe and to control. It is not equigenerated and thus cannot have linear resolution, for instance. This is due to the lack of symmetry in the complete part of $\Lin(I)$. On the other hand one has
$$\LIN(I)=(x_1,x_2,x_3)^3_{\le(2,2,2,0,0)}+(x_1x_2y_1,\ x_1^2y_1,\ x_2x_3y_2,\ x_1x_3y_2,\ x_1x_2y_2),$$
which has a very symmetric complete part, whose nice properties are inherited by
$$\sqrt{\LIN(I)}=(x_1x_2,\ x_1x_3,\ x_2x_3,\ x_1y_1).$$
 Computer calculations show that $\sqrt{\LIN(I)}$ in this example has linear quotients, and hence linear resolution.
Notice moreover that $\sqrt{\LIN(I)}$ \emph{looks like} the linearization of something but ``lacks a piece'', for instance $x_2y_1$ or $x_3y_1$.
\end{example}

\begin{example}
Take now $I=(x_1^2x_2^2,x_2^2x_3^2)\subset \K[x_1,x_2,x_3]$. The ideal
$$\sqrt{\LIN(I)}=(x_1x_2,\ x_1x_3,\ x_2x_3)$$
has again linear quotients, and this time it looks even prettier, with no generators involving $y$-variables. More precisely, one can compute the radical ideal of $\LIN(I)$ by taking the radicals of the generators, and one realizes that the radicals of the generators of the last part of $\LIN(I)$---which are the ones involving the $y$-variables---turn out to be redundant. Notice that in this case $\LIN(I)=\Lin(I)$.
\end{example}

\begin{notation}
As in the rest of the section, let $d$ be the degree of all the minimal monomial generators of $I$ and let $M$ be the largest exponent occurring in these generators. Write
$$d=aM+b,\qquad\text{with $a,b\in\N$ and $b<M$,}$$
in the ``Euclidean way''. As motivated in the discussions above, we assume that
\begin{equation}\label{assumptioneuclid}
a\ge1\qquad\text{and}\qquad M\ge2.
\end{equation}
Moreover, for a non-negative integer $b\in\N$, we write
$$\sign(b)=
\begin{cases}
1&\text{if $b>0$,}\\
0&\text{if $b=0$.}
\end{cases}$$
\end{notation}

In the following, recall  that the support $\Supp(u)$ of a monomial~$u$ is the set consisting of the variables which divide~$u$.

\begin{proposition}
The ideal $\sqrt{\LIN(I)}$ is generated by all squarefree monomials of degree $a+\sign(b)$ in the variables $x_1,\dots,x_n$ and by the monomials $\sqrt{\frac{f_j}{x_k}}y_j$, where $f_j$ is a minimal generator of $I$, $x_k$ occurs with exponent~$1$ in $f_j$, and all the other variables occur with exponent~$M$. Notice that we might possibly have such ``pathological'' generators, that is, generators with exactly one variable that occurs with exponent $1$ and all other variables with exponent~$M$, only when $b=1$. Let $p$ be the number of ``pathological'' generators of~$I$. Then
$$\#G\big(\sqrt{\LIN(I)}\big)=\binom n{a+\sign(b)}+p,$$
and all the generators have the same degree $a+\sign(b)$.
\end{proposition}

\begin{proof}
When we take the radical of the complete part of $\LIN(I)$, we find all squarefree monomials of degree $a+\sign(b)$ in the variables $x_1,\dots,x_n$. Let's see what happens on the other hand to the last part:
$$\sqrt{\text{last part of $\LIN(I)$}}=\big(\sqrt{f_jy_j/x_k}\mid \text{$x_k$ divides $f_j$, }k\in\{1,\dots,n\}, j\in\{1,\dots,m\}\big).$$
Of course one has $\Supp(\sqrt u)=\Supp(u)$ for any monomial $u$. Moreover, 
\begin{equation}\label{supportivari}
\#\Supp\Big(\frac{f_j}{x_k}\Big)=
\begin{cases}
\#\Supp(f_j)&\text{if $x_k$ divides $\frac{f_j}{x_k}$,}\\
\#\Supp(f_j)-1&\text{otherwise.}
\end{cases}
\end{equation}
If $\#\Supp(f_j/x_k)\ge a+\sign(b)$, then the generator $\sqrt{\frac{f_jy_j}{x_k}}=\sqrt{\frac{f_j}{x_k}}y_j$ is redundant because it's a multiple of a generator coming from the complete part.
Hence we are only interested in the generators such that $\#\Supp(f_j/x_k)<a+\sign(b)$.

Notice furthermore that we  have $\#\Supp(f_j)\ge a$: indeed, assume that $\#\Supp(f_j)=a'<a$. Then one would have $\deg(f_j)\le a'M<aM\le d$, a contradiction. So, by~\eqref{supportivari}, $\#\Supp(f_j/x_k)<a+\sign(b)$ actually happens when
$$\begin{cases}
\#\Supp(f_j)<a+\sign(b)&\text{if $x_k$ divides $\frac{f_j}{x_k}$,}\\
\#\Supp(f_j)-1<a+\sign(b)&\text{otherwise.}
\end{cases}$$
The first case can happen only if $\#\Supp(f_j)=a$ and $\sign(b)=1$. The second case can happen only when $\#\Supp(f_j)=a$ and for any $b$, or when $\#\Supp(f_j)=a+1$ and $\sign(b)=1$. Notice we cannot have simultaneously $\#\Supp(f_j)=a$ and $\sign(b)=1$, because then we would have $\deg(f_j)\le aM<aM+b=d$, a contradition.   So actually only the second case above can happen, when we have either $\#\Supp(f_j)=a$ and $b=0$, or when $\#\Supp(f_j)=a+1$ and $\sign(b)=1$. In short, when $\#\Supp(f_j)=a+\sign(b)$.

%To recap, if we have $f_j$ and $x_k$ such that $x_k$ occurs with exponent~$1$ in~$f_j$, and such that $\#\Supp(f_j)=a+\sign(b)$, then we get $\#\Supp(f_j/x_k)=a+\sign(b)-1$ and the generator $\sqrt{\frac{f_j}{x_k}}y_j$ is not redundant.

In conclusion, $\sqrt{\LIN(I)}$ is generated by $(x_1,\dots,x_n)^{a+\sign(b)}_\sqf$ and by the monomials $\sqrt{\frac{f_j}{x_k}}y_j$ where $x_k$ occurs in $f_j$ with exponent~$1$ and such that $\#\Supp(f_j)=a+\sign(b)$.  Now, if we have something as above, with $x_k$ occurring with exponent~$1$ and such that $\#\Supp(f_j)=a+\sign(b)$, then it means that $x_k$ is the only variable with exponent~$1$ inside $f_j$, and all the rest have exponent $M$. This is because we are assuming $M\ge 2$. If we have another variable with exponent~$<M$, then we can just ``move'' one exponent~$1$ with the other exponent $<M$, so that we get a monomial of degree $d$, with variables occurring all with exponent~$\le M$ and with support strictly smaller than that of $f_j$. So in other words $a+\sign(b)\le\#\Supp(f_j/x_k)$, and the generator $\sqrt{\frac{f_j}{x_k}}y_j$ turns out to be redundant.
\end{proof}

In the following we use the same notation as in Corollary~\ref{solitocorollariork} for the numbers~$r_k$, in view of Corollary~\ref{bnumsradical}.

\begin{theorem}\label{radhaslinearquotients}
The ideal $\sqrt{\LIN(I)}$ has linear quotients. Let us order the generators so that we first have all squarefree monomials of degree $a+\sign(b)$  in decreasing lexicographic order and then we have the radicals of the ``pathological'' generators of $I$. Then the numbers $r_k$ which come from those generators range between $0$ and $n-a-\sign(b)$. After that, in case $b=1$, we still have $p$ colon ideals (possibly with $p=0$). Order the pathological generators 
$$\sqrt{\frac{f_{j_1}}{x_{k_1}}}y_{j_1},\ \dots,\ \sqrt{\frac{f_{j_p}}{x_{k_p}}}y_{j_p}.$$
Let $r_k$ be the number of generators of the colon ideal by $\sqrt{\frac{f_{j_\ell}}{x_{k_\ell}}}y_{j_\ell}$. Define
$$t_\ell:=\#\Big\{s \in\{1,\dots,\ell-1\}\mid \frac{f_{j_s}}{x_{k_s}}=\frac{f_{j_\ell}}{x_{k_\ell}}\Big\}.$$ 
Then $r_k=n-a+t_\ell$.
\end{theorem}

\begin{proof}
The first numbers $r_k$ behave as in the case studied for the linearization of a squarefree ideal, for which one can see Proposition~\ref{rkpartecompletasqf}. When we take the colon by $\sqrt{\frac{f_{j_\ell}}{x_{k_\ell}}}y_{j_\ell}$, we get for sure as generators at least each $x_i\notin\Supp(f_{j_\ell}/x_{k_\ell})$, and there are $n-a$ of them. Additionally, if we have $\frac{f_{j_s}}{x_{k_s}}=\frac{f_{j_\ell}}{x_{k_\ell}}$ for some $s<\ell$, then $\sqrt{f_{j_s}/x_{k_s}}=\sqrt{f_{j_\ell}/x_{k_\ell}}$. So, when taking the colon we get the extra generator $y_{j_s}$.
\end{proof}

\begin{corollary}\label{bnumsradical}
With the same notation as above, we have
$$\beta_i(\sqrt{\LIN(I)})=\binom{i+a+\sign(b)-1}{a+\sign(b)-1}\binom n{i+a+\sign(b)}+\sum_{\ell=1}^p\binom{n-a+t_\ell}i.$$
\end{corollary}

\begin{proof}
For the complete part see Remark~\ref{bettinbsveronese}. The rest follows from the previous theorem and Corollary~\ref{solitocorollariork}.
\end{proof}

%%%%%%%%%%%%%%%%%%%%%%%%%%%%
%%%%%%%%%%%%%%%%%%%%%%%%%%%%

\section{Linearization in the squarefree case}\label{linearsquarefree}

In this section $\K$ may be any field, not necessarily of characteristic zero. Indeed, the properties here discussed are purely combinatorial and do not depend on the characteristic of $\K$. 

\begin{notation}
Let $I$ be a squarefree monomial ideal in $S:=\K[x_1,\dots,x_n]$, with minimal system of monomial generators $G(I)=\{f_1,\dots,f_m\}$, where $\deg(f_j)=d$ for all $j$.
\end{notation}

The results in this section are written for $\LIN(I)$, but it's very easy to adapt them to $\Lin(I)$, as explained in the following.

\begin{remark}\label{differenzaLinLIN}
If we specify the definition of $\LIN$ to this case, we get, inside the polynomial ring $R:=\K[x_1,\dots,x_n,y_1,\dots,y_m]$, the ideal
\begin{align*}
\LIN(I)&=\big(x_{i_1}x_{i_2}\dots x_{i_d}\mid 1\le i_1<i_2<\dots<i_d\le n\big)\\
&\quad+\big(f_jy_j/x_k \mid \text{$x_k$ divides $f_j$},\ k=1,\dots,n,\ j=1,\dots,m\big).
\end{align*}
That is, in the complete part we have all squarefree monomials of degree $d$ in the $x_i$'s. In the last part, for each generator\til$f_j$ of $I$ we add a new variable\til$y_j$ and $d$ generators $f_jy_j/x_k$, where we replace the only occurrence of each\til$x_k$ in the support of $f_j$ by $y_j$. 
The difference between $\Lin(I)$ and $\LIN(I)$ is not  significant in the squarefree case. 
In the complete part of $\Lin(I)$ we don't have all squarefree monomials of degree $d$ in all the $n$ variables, but instead we have all squarefree monomials of degree $d$ in the variables that appear in the generators of $I$. So, what one needs to do in order to write a version for $\Lin(I)$ of the results in this section is to replace $x_1,\dots,x_n$ by the variables which are actually used by the generators of $I$. Or alternatively one can assume that each of the variables $x_1,\dots,x_n$ appears in some generator of $I$, in which case (in the squarefree situation!) we have $\LIN(I)=\Lin(I)$. As a side note, observe that the slight difference makes $\LIN(I)$ a direct generalization of the Booth--Lueker ideals associated to the graphs introduced in Definition~\ref{defBL}.
\end{remark}

\begin{notation}
Instead of writing $v=(1,\dots,1)$ in the formula $J_{\le v}$ introduced in Notation~\ref{notationcropping}, for the squarefree case we simply write $J_\sqf$.
\end{notation}

\begin{remark}\label{remarksqlin}
One can see that $I$ is squarefree if and only if $\Lin(I)$ is squarefree. Both implications are easy: the ``only if'' implication is clear by construction and the ``if'' one follows from the fact that the complete part of $\Lin(I)$ is squarefree if and only if $I$ is. The last part of $\Lin(I)$ would in general not be enough to detect the squarefreeness of $I$, if for instance $I$ had a generator like $x_1^2$. But for degree~$d>2$  it's true that $I$ is squarefree if and only if the last part of $\Lin(I)$ is. Same goes for $\LIN$.
\end{remark}

The next is an example of the abovementioned slight difference between $\LIN(I)$ and $\Lin(I)$ in the squarefree case, and of how one can obtain from the results about~$\LIN(I)$ the corresponding ones for~$\Lin(I)$. The number of minimal generators of $\LIN(I)$ is
$$\#G(\LIN(I))=\binom nd+md.$$
Let $c:=\#\bigcup_{j=1}^m\Supp(f_j)$. The number of minimal generators of $\Lin(I)$ is
$$\#G(\Lin(I))=\binom cd+md.$$

\subsection{Betti numbers of $\LIN(I)$}

As proven in Theorem~\ref{mainthm}, $\Lin(I)$ and $\LIN(I)$ have linear quotients for any equigenerated ideal~$I$, so in particular for squarefree equigenerated ideals. In this section we provide explicit formulas for the Betti numbers of $\LIN(I)$ in case $I$ is squarefree. As already remarked above, one can turn the results here obtained for $\LIN(I)$ into analogous ones for $\Lin(I)$ simply by replacing $n$ by the number of variables, among $x_1,\dots,x_n$, that actually appear in the minimal monomial generators of~$I$.

In virtue of Corollary~\ref{bnumsradical} and Theorem~\ref{mainthm}, 
we can use the formula
$$\beta_i\big(\LIN(I)\big)=\sum_{k=1}^{\#G(\LIN(I))}\binom{r_k}i,\qquad \text{for $i\ge 0$,}$$
where $r_k$ is the number of generators of $(g_1,\dots,g_{k-1}):g_k$ and where the polynomials $g_j$ are the generators of $\LIN(I)$ as listed in Theorem~\ref{mainthm}. We determine the numbers~$r_k$ in Proposition~\ref{rkpartecompletasqf} and Proposition~\ref{lastpartrk}.

\begin{example}\label{esempionesqfree}
Consider the ideal 
$$I=(x_1x_2x_3,\ x_1x_2x_4,\ x_1x_2x_5)\quad\subset\quad S=\K[x_1,\dots,x_5].$$
The hypergraph corresponding to this ideal (see Section~\ref{backgroundhypgraphs}) consists of three triangles that share the common edge $x_1x_2$ (which has codimension\til$1$). The $*$-linearization of $I$ lives in $R=\K[x_1,\dots,x_5,y_1,y_2,y_3]$, and it's the ideal
\begin{align*}
\LIN(I)&=(x_1x_2x_3,\ x_1x_2x_4,\dots,x_2x_4x_5,\ x_3x_4x_5)\\
&\quad+(y_1x_2x_3,\ x_1y_1x_3,\ x_1x_2y_1,\dots,y_3x_2x_5,\ x_1y_3x_5,\ x_1x_2y_3),
\end{align*}
with $\binom nd+md=\binom53+3\times3=19$ generators. In fact, in this example one has $\Lin(I)=\LIN(I)$. One can compute the Betti table of this ideal,
$$\beta(\LIN(I))=
\begin{array}{c|cccccc}
&0&1&2&3&4&5\\
\hline
3&19&45&43&21&6&1.\\
\end{array}$$
Next we compute the colon ideals $J_k:=(g_1,\dots,g_{k-1}):g_k$. For $g_k$ in the complete part of $\LIN(I)$ we get
\begin{align*}
J_1&=(0)\\
J_2&=(x_3)\\
J_3&=(x_3,\  x_4)\\
J_4&=(x_2)\\
J_5&=(x_2,\  x_4)\\
J_6&=(x_2,\  x_3)\\
J_7&=(x_1)\\
J_8&=(x_1,\  x_4)\\
J_9&=(x_1,\  x_3)\\
J_{10}&=(x_1,\  x_2)
\end{align*}
and, after that for $g_k$ in the last part we get
\begin{align*}
J_{11}&=(x_1,\  x_4,\  x_5)\\
J_{12}&=(x_2,\  x_4,\  x_5)\\
J_{13}&=(x_3,\  x_4,\  x_5)\\
J_{14}&=(x_1,\  x_3,\  x_5)\\
J_{15}&=(x_2,\  x_3,\  x_5)\\
J_{16}&=(x_3,\  x_4,\  x_5,\ y_1)\\
J_{17}&=(x_1,\  x_3,\  x_4)\\
J_{18}&=(x_2,\  x_3,\  x_4)\\
J_{19}&=(x_3,\  x_4,\  x_5,\ y_1,\ y_2).
\end{align*}
%\begin{verbatim}
%ideal(0)
%ideal(x[3])
%ideal(x[4],  x[3])
%ideal(x[2])
%ideal(x[4],  x[2])
%ideal(x[3],  x[2])
%ideal(x[1])
%ideal(x[4],  x[1])
%ideal(x[3],  x[1])
%ideal(x[2],  x[1])
%ideal(x[5],  x[4],  x[1])
%ideal(x[5],  x[4],  x[2])
%ideal(x[5],  x[4],  x[3])
%ideal(x[5],  x[3],  x[1])
%ideal(x[5],  x[3],  x[2])
%ideal(y[1],  x[5],  x[4],  x[3])
%ideal(x[4],  x[3],  x[1])
%ideal(x[4],  x[3],  x[2])
%ideal(y[2],  y[1],  x[5],  x[4],  x[3])
%\end{verbatim}
This example is quite small but we can observe some facts that hold in general:
the colon ideals $J_k$ with $g_k$ in the complete part are very regular, they clearly don't depend on the generators of $I$, and they have nothing to do with the variables $y_j$. They have at most $n-d$ generators, in this case $5-3=2$. On the other hand the ideals $J_k$ with $g_k$ in the last part of $\LIN(I)$ have more generators, at least $n-d+1$, and some of these ideals happen to have additional generators, consisting of some variables $y_j$.
\end{example}

\begin{lemma}\label{randpeeva}
Let $m_1,\dots,m_{\binom nd}\in \K[x_1,\dots,x_n]$ be all the squarefree monomials of degree $d$ in $n$ variables. Assume these monomials are ordered in decreasing lexicographic order.  For each $k$, denote by $r_k$ the number of minimal monomial generators of $(m_1,\dots,m_{k-1}):m_k$. For any monomial $m$, denote  $\max(m):=\max\{i\mid\text{$x_i$ divides $m$}\}$. Then 
$$r_k=\max(m_k)-d.$$
\end{lemma}

\begin{proof}
For $k=1$ this is clear. Let us  consider $m_k$ for $k>1$. The  variables $x_i$ with $i<\max(m_k)$ which are not in the support of $m_k$ are 
$$\max(m_k)-1-(d-1)=\max(m_k)-d.$$
Denote $j:=\max(m_k)$. For each $i<j$ such that $x_i\notin\Supp(m_k)$, we have that $\frac{x_i}{x_j}m_k$ has degree $d$ and comes before $m_k$ in the lexicographic order. So $(m_1,\dots,m_{k-1}):m_k$ has as generators all such variables $x_i$. Let us see that there cannot be more generators. There might be a monomial $m$ that comes lexicographically before $m_k$ and such that $\frac m{\gcd(m,m_k)}$ does not consist of just one variable, but in that case this monomial $m$ is divided by some variable $x_i$ with $i<j$, or else $m$ would not be smaller than $m_k$ in the lexicographic order.
\end{proof}

\begin{proposition}\label{rkpartecompletasqf}
The colon ideals  coming from the complete part of $\LIN(I)$ behave as follows. The numbers $r_1,r_2,\dots,r_{\binom nd}$ range between $0$ and $n-d$. Number $j\in\{0,\dots,n-d\}$ occurs as $r_k$ for
$$\binom{j+d-1}{d-1}$$
values of $k$.
\end{proposition}

\begin{proof}
This follows immediately from Lemma~\ref{randpeeva}. For each $k$ such that $\max(m_k)=j+d$, we have that $j=r_k=\binom{j+d-1}{d-1}$ is the number of squarefree monomials of degree $d$ with maximum equal to $j+d$.
\end{proof}

\begin{remark}\label{bettinbsveronese}
We recover the well-known Betti numbers of the squarefree power $C:=(x_1,\dots,x_n)^d_{\sqf}$ as
\begin{align*}
\beta_i(C)&=\sum_{k=1}^{\binom nd}\binom{r_k}i\\
&=\sum_{j=0}^{n-d}\binom{j+d-1}{d-1}\binom ji\\
&\stackrel{(*)}=\binom{i+d-1}{d-1}\sum_{j=0}^{n-d}\binom{j+d-1}{i+d-1}\\
&=\binom{i+d-1}{d-1}\sum_{j=d-1}^{n-1}\binom j{i+d-1}\\
&=\binom{i+d-1}{d-1}\binom n{i+d},
\end{align*}
where in~$(*)$ we used that, by direct calculation,
$$\binom{j+d-1}{d-1}\binom ji=\binom{i+d-1}{d-1}\binom{j+d-1}{i+d-1}$$
 and in the last equality the formula
$$\sum_{j=d-1}^{n-1}\binom j{i+d-1}=\binom n{i+d},$$
which is a specialization of the identity~(11) in Section~1.2.6 of~\cite{K73}.
\end{remark}

In order to describe what happens to the rest of the colon ideals, namely those that arise from the last part of $\LIN(I)$, we need some definitions. There probably already is a better terminology   in the literature to express these things, but due to my lack of knowledge I introduce the following terms.

\begin{definition}\label{edgemultiplicity}
We call a monomial $u=x_{i_1}\cdots x_{i_{d-1}}$, with $i_1<\dots<i_{d-1}$, a \tbs{$(d-1)$-edge of $I$} if $u$ divides some generator of $I$. Let us call \tbs{multiplicity} of a $(d-1)$-edge $u$ of $I$ the number
$$\mult(u):=\#\{f_i\in G(I)\mid\text{$u$ divides $f_i$}\},$$
where $G(I)=\{f_1,\dots,f_m\}$ as in all the rest of the section. Call a \tbs{$j$-cluster} a set of cardinality $j$ consisting of generators of $I$ that are divided by a same $(d-1)$-edge $u$.
\end{definition}

Consider a graph and the corresponding edge ideal. Then a $1$-edge would be a single variable dividing some generator of the edge ideal, that is, it would correspond to a vertex which is touched by some  edge---with the ordinary meaning of the word ``edge''---of the graph. The multiplicity of this $1$-edge/vertex means the number of edges touching it, namely the degree of the vertex. So \emph{the multiplicity here defined generalises the graph-theoretical notion of degree}.

\begin{example}
Continuing Example~\ref{esempionesqfree}, for instance $x_1x_5$ is a $2$-edge of $I$, since it divides the generator $f_3=x_1x_2x_5$. The multiplicity of the $2$-edge $x_1x_5$ is $1$, since $x_1x_5$ divides only the generator $f_3$. Another $2$-edge is $x_1x_2$, and this one has multiplicity $3$, as it divides all three generators. The presence of such a $2$-edge of multiplicity $3$ turns out to be the reason why in Example~\ref{esempionesqfree} we have a quotient ideal with an extra $y$-generator and a quotient ideal with two extra $y$-generators.
To conclude this example, notice that the set $\{x_1x_2x_4,x_1x_2x_5\}$ is a $2$-cluster because all of its elements are divided by the same $2$-edge $x_1x_2$. Similarly, the set $G(I)$  is a $3$-cluster.
\end{example}

\begin{proposition}\label{lastpartrk}
The numbers $r_k$ coming from the last part of $\LIN(I)$ range from $n-d+1$ and above.
For each integer $j\ge2$, consider the maximal $j$-clusters, that is, $j$-clusters that are not part of a $(j+1)$-cluster. For each maximal $j$-cluster there is a colon ideal with $n-d+2$ generators, a colon ideal with $n-d+3$ generators,\dots, up to a colon ideal with $n-d+j$ generators. Any maximal $j$-cluster has its own set of such ideals. All other colon ideals have $n- d+1$ generators.
\end{proposition}

\begin{proof}
All the colon ideals of the form $(g_1,\dots,g_{k-1}):g_k$ with $g_k=\frac{f_\ell}{x_i}y_\ell$ in the last part of $\LIN(I)$ have at least $n-d+1$ generators. This is because among those generators we have at least each $x_\ell\notin\Supp\big(\frac{f_\ell}{x_i}\big)$, and those are exactly $n-d+1$.

Now, each of these ideals might also have some ``extra'' $y$-generators. For each $j\ge2$, each maximal $j$-cluster contributes to the list of ideals with an ideal with one extra $y$-generator, an ideal with two extra $y$-generators, \dots, up to an ideal with $j-1$ extra $y$-generators. The rest of the quotient ideals coming from the last part of $\LIN(I)$ have $n-d+1$ generators, and they correspond to the maximal $1$-clusters. This is  clear by Remark~\ref{remarkcolon}, as follows. Write $f_{\ell_1},\dots,f_{\ell_j}$ for all the generators inside a maximal $j$-cluster. So then we have a $(d-1)$-edge $u$ shared by  them all, so that $f_{\ell_p}=ux_{i_p}$ for all $p\in\{1,\dots,j\}$. In the last part of $\LIN(I)$ we have generators $\frac{f_{\ell_p}}{x_i}y_{\ell_p}$, for each $x_i$ dividing each $f_{\ell_p}$, namely each $x_i$ dividing $u$ and $x_{i_p}$. Then when taking the colon by $\frac{f_{\ell_p}}{x_i}y_{\ell_p}$ we get exactly the number of generators stated above.
\end{proof}

\begin{corollary}\label{bnumbersforlin}
Let us denote by $C_j\in\N$ the number of maximal $j$-clusters, that is, $j$-clusters that are not part of a $(j+1)$-cluster. Then
\begin{align*}
\beta_i\big(\LIN(I)\big)&=\binom{i+d-1}{d-1}\binom n{i+d}
+\binom{n-d+1}i\bigg(md-\sum_{j\ge2}(j-1)C_j\bigg)\\
&\quad+\sum_{j\ge2}C_j\sum_{k=2}^j\binom{n-d+k}i.
\end{align*}
\end{corollary}

\begin{corollary}
Let $N:=\max\{j\mid C_j\ne 0\}$. Then 
\begin{align*}
\pd_R(\LIN(I))&=n-d+N,\\
\depth(R/\LIN(I))&=m+d-N-1.
\end{align*}
\end{corollary}

\begin{proof}
The largest $r_k$ occurring in the sum is $n-d+N$, so the last non-zero  binomial coefficient
$$\binom{r_k}{i}=\binom{n-d+N}{i}$$ corresponds to $i=n-d+N$. So  we get the desired formula we wanted for the projective dimension. As for the depth, we just use the Auslander--Buchsbaum formula (see Formula~15.3 of~\cite{Pe}):
\begin{align*}
\depth(R/\LIN(I))&=\depth(R)-\pd_R(R/\LIN(I))\\
&=n+m-(n-d+N+1).
\end{align*}
\end{proof}

%\begin{question}
%What can I say about the equivalence between having linear quotients and the Alexander dual being shellable? How does it show in the case of $\Lin(I)$?
%\end{question}

%%%%%%%%%%%%%%%%%%%%%%%%%%%%%%%%

\subsection{Conceptual proof via polarizations}\label{gunnarsproof}

What follows is a more conceptual proof, in the squarefree case, that $\LIN(I)$ has linear resolution, a result which we already have as a special case of Corollary~\ref{Linlinres}, and that the Betti numbers only depend on the multiplicities of the $(d-1)$-edges, as defined in Definition~\ref{edgemultiplicity}. This proof and the notion of separation involved in it were very kindly explained to me by Gunnar Fl\o ystad.

\begin{notation}
For any  subset $\sigma\se[n]:=\{1,\dots,n\}$,  denote $x^\sigma:=\prod_{i\in\sigma} x_i$ and $\mm^\sigma:=(x_i\mid i\in\sigma)$.
\end{notation}

Given a squarefree monomial ideal $I=(x^{\sigma_1},\dots,x^{\sigma_s})\se \K[x_1,\dots,x_n]$, recall that the \tbs{Alexander dual} of $I$ is the ideal
\begin{align*}
I^\vee&:=\mm^{\sigma_1}\cap\dots\cap\mm^{\sigma_s}\\
&\phantom{:}=(\text{monomials with non-trivial common divisor with}\\
&\phantom{:=}\quad \text{every monomial (equivalently, generator) of $I$}).
\end{align*}
For other equivalent descriptions of $I^\vee$ and additional information, see for instance Section~62 of~\cite{Pe} or Section~1.5.2 of~\cite{HH}.

Let $R$ be a commutative ring and let $M$ be an $R$-module. Recall that an element $a\in M$ is \tbs{$M$-regular} if the only $m\in M$ such that $am=0$ is $m=0$. A sequence $a_1,\dots,a_n\in R$ is an \tbs{$M$-regular sequence} if the following hold:
\begin{itemize}
\item $a_i$ is $M/(a_1,\dots,a_{i-1})M$-regular for all $i\in\{1,\dots,n\}$;
\item $M/(a_1,\dots,a_n)M\ne0$.
\end{itemize}

%Next we recall an ubiquitous notion in commutative algebra, without specifying all the details involved because it's just instrumental for our purposes.
%
%\begin{definition}
%For a Noetherian local ring $R$, a finitely generated $R$-module $M\ne0$ is a \tbs{Cohen--Macaulay module} if $\depth(M)=\dim(M)$. If $R$ is any Notherian ring, $M$ is called a \tbs{Cohen--Macaualy module} if the localization $M_\mm$ is Cohen--Macaulay  as defined in the local case, for any maximal ideal $\mm$ in the support of $M$. If $R$ is Cohen--Macaualy as an $R$-module, then we say that $R$ is a \tbs{Cohen--Macaulay ring}.
%\end{definition}

The combination of the  following two theorems, which are very well known, constitutes the main tool for the proof presented in this section. For the notion of Cohen--Macaulay ring, see Definition\til2.1.1 of~\cite{BH93}. For a proof of the first theorem, see for instance Theorem~2.1.3 of~\cite{BH93}. For the second theorem see for instance Corollary~62.9 of~\cite{Pe} or the original version, which is Theorem~3 of~\cite{EaRe}.

\begin{theorem}\label{CMquotientCM}
Let $R$ be a Notherian ring and $a_1,\dots,a_n$ be a regular sequence in $R$. If $R$ is a Cohen--Macaulay ring, then $R/(a_1,\dots,a_n)$ is a Cohen--Macaualy ring.
\end{theorem}

\begin{theorem}[Eagon--Reiner,~\cite{EaRe}]\label{eigonreiner}
For a squarefree monomial ideal $I\se S$, the following are equivalent:
\begin{itemize}
\item $I$ has linear resolution (see Definition~\ref{deflinearreso});
\item $S/I^\vee$ is Cohen--Macaulay.
\end{itemize}
\end{theorem}

The very classical notion of \emph{polarization} has been often used in commutative algebra and related fields, in particular to reduce the study of homological properties of any monomial ideal to the case of squarefree monomial ideals. It was originally used by Hartshorne in his proof of the connectedness of the Hilbert scheme, see Chapter~4 of Hartshorne's paper~\cite{Hart}. See Section~21 of~\cite{Pe} for additional information. Later it became a standard tool in commutative algebra thanks to the work of Hochster. The notion in the next definition, fundamental for the proof presented in this section, is a generalization of the classical polarization. It probably first appeared in~\cite{Yana} and  a systematic study of it is done in the recent paper~\cite{AFL} for powers of graded maximal ideals. Gunnar Fl\o ystad, the second author of that paper, showed me how to prove that~$\LIN(I)$ has a linear resolution using this framework.

\begin{definition}
Let $p\:R'\to R$ be a surjection of finite sets with the cardinality of $R'$ one more than that of $R$. Let $r_1$ and $r_2$ be two distinct elements of $R'$ such that $p(r_1)=p(r_2)$. Denote for short $\K[x_R]:=\K[x_i\mid i\in R]$. Let $I$ be a monomial ideal in~$\K[x_R]$ and $J$ a monomial ideal in $\K[x_{R'}]$. We say that $J$ is a \tbs{separation of $I$} if the following conditions hold:
\begin{enumerate}
\item[(1)] The ideal $I$ is the image of $J$ by the map $\K[x_{R'}]\to \K[x_R]$ induced by $p$.
\item[(2)] Both the variables $x_{r_1}$ and $x_{r_2}$ occur in some minimal generators of $J$ (usually in distinct generators).
\item[(3)] The variable difference $x_{r_1}-x_{r_2}$ is a non-zero-divisor in the quotient ring $\K[x_{R'}]/J$.
\end{enumerate}
More generally, if $p\:R'\to R$ is a surjection of finite sets and $I\se \K[x_R]$ and $J\se\K[x_{R'}]$ are monomial ideals such that $J$ is obtained by a succession of separations of $I$, we also call $J$ a \tbs{separation of $I$}. If $J$ is squarefree and a separation of $I$, then we say that $J$ is a \tbs{polarization of} $I$.
\end{definition}

\subsubsection{Preliminary constructions}

Consider the following ideal, generated by all squarefree monomials of degree~$d$ inside the polynomial ring $\K[x_0,x_1,\dots,x_n]$, with $d< n$:
$$J:=(x_0,x_1,\dots,x_n)^d_{\sqf}.$$
The ideal $J$ is the \emph{squarefree (Veronese) $n$-th power} of $(x_0,\dots,x_n)$, and it is a well-known fact that $J$ has linear resolution and is Cohen--Macaulay. Starting from this ideal $J$, we will perform some separations and take  Alexander duals, and eventually get to~$\LIN(I)$. (This might perhaps be surprising, since  $J$ looks ``more symmetric'' than $\LIN(I)$: indeed, quoting the authors of~\cite{AFL}, \emph{``a polarization is somehow a way of breaking this symmetry,
but still keeping the homological properties''}.) We may write
$$J=(x_1,\dots,x_n)^d_{\sqf}+x_0(x_1,\dots,x_n)^{d-1}_{\sqf}.$$
The ideal $J$ can be separated to
\begin{align*}
H&:=(x_1,\dots,x_n)^d_{\sqf}\\
&\ \quad+\big(y_\ii\ x^\ii\mid\ii=\{i_1,\dots,i_{d-1}\},1\le i_1<\dots<i_{d-1}\le n\big),
\end{align*}
where $x^\ii=x_{i_1}\cdots x_{i_{d-1}}$,
in the polynomial ring $\K\big[x_1,\dots,x_n,y_\ii\mid\ii\in\binom{[n]}{d-1}\big]$ where the new variables $y_\ii$'s are indexed on all the $(d-1)$-subsets of $[n]=\{1,\dots,n\}$. We denote the ring $\K[x_1,\dots,x_n,y_\ii\text{'s}]$ for short. Order all these $(d-1)$-subsets as $\ii_1,\ii_2,\dots,\ii_{\binom n{d-1}}$.
Then $\K[x_0,x_1,\dots,x_n]/J$ is obtained from $\K[x_1,\dots,x_n,y_\ii\text{'s}]/H$ by dividing out by the variable differences 
\begin{equation}\label{vardiffconst}
y_{\ii_1}-y_{\ii_2},\ 
y_{\ii_2}-y_{\ii_3},\ \dots,\ 
y_{\ii_{\binom n{d-1}-1}}-y_{\ii_{\binom n{d-1}}},
\end{equation}
so that all $y_\ii$'s are identified to a single variable~$x_0$. 

\begin{lemma}\label{regsequno}
The variables differences in~\eqref{vardiffconst} form a $\K[x_1,\dots,x_n,y_\ii\text{'s}]/H$-regular sequence.
\end{lemma}

\begin{proof}
We start by showing that $y_{\ii_1}-y_{\ii_2}$ is a regular element in the ring $\K[x_1,\dots,x_n,y_\ii\text{'s}]/H$. Equivalently, if some polynomial of 
$\K[x_1,\dots,x_n,y_\ii\text{'s}]$ multiplies  $y_{\ii_1}-y_{\ii_2}$ in $H$, then we want to show that this polynomial belongs to $H$. In fact, since $H$ is a squarefree monomial ideal, it's enough to show that if we have $u(y_{\ii_1}-y_{\ii_2})\in H$ for a monomial $u$, then $u\in H$. Since $H$ is monomial, $u(y_{\ii_1}-y_{\ii_2})\in H$ implies that $uy_{\ii_1}\in H$ and $uy_{\ii_2}\in H$, so that for some generators $g_1,g_2\in G(H)$ we have that $g_1$ divides $uy_{\ii_1}$ and $g_2$ divides $uy_{\ii_2}$. If $g_1$ or $g_2$ divides $u$, then we're done. If not, then it means that $g_1=y_{\ii_1}x^{\ii_1}$ and $g_2=y_{\ii_2}x^{\ii_2}$. So then $\lcm(x^{\ii_1},x^{\ii_2})$ is a monomial in at least $d$~ variables, therefore in $H$, dividing $u$, and then we are done.

To continue the proof, we notice that after taking the quotient by some differences $y_{\ii_1}-y_{\ii_2},\ 
y_{\ii_2}-y_{\ii_3},\ \dots,\ 
y_{\ii_{j-1}}-y_{\ii_j}$ we get a polynomial ring in less variables and an ideal $\cl H$ whose generators are the same as $H$, except for identifications of some $y$-variables. Those variables are not involved in proving that $y_{\ii_j}-y_{\ii_{j+1}}$ is a regular element in the new ring, hence we can simply iterate the argument above.
\end{proof}

Hence $H$ has a linear resolution, too. The Alexander dual of $H$ is 
$$H^\vee=(x_1,\dots,x_n)^{n-d+2}_{\sqf}+\Big(\frac{x_1\cdots x_n}{x^\ii}y_\ii\mid\ii\in\binom{[n]}{d-1}\Big).$$
By Theorem~\ref{eigonreiner}, the quotient $\K[x_1,\dots,x_n,y_\ii\text{'s}]/H^\vee$ is a Cohen--Macaulay ring, since $H$ has linear resolution. Next, we will take a further quotient of the ring $\K[x_1,\dots,x_n,y_\ii\text{'s}]/H^\vee$ by a regular sequence, thus preserving Cohen--Macaulayness. 

\begin{lemma}
The differences
$$y_{\ii_1}-1,\ y_{\ii_2}-1,\dots,\ y_{\ii_{\binom n{d-1}}}-1$$
form a $\K[x_1,\dots,x_n,y_\ii\text{'s}]/H^\vee$-regular sequence.
\end{lemma}

\begin{proof}
Similarly to the case of Lemma~\ref{regsequno}, we start by showing that if we have $u(y_{\ii_1}-1)\in H^\vee$ for some monomial~$u$, then $u\in H^\vee$. This is clear because $H^\vee$ is a monomial ideal and $u(y_{\ii_1}-1)=uy_{\ii_1}-u$. Again, as in the proof of Lemma~\ref{regsequno}, we can iterate this and we get that the variable differences in the statement form a regular sequence.
\end{proof}

Let $U\subseteq\binom{[n]}{d-1}$ be a set of subsets of $[n]$ of cardinality~$d-1$. We quotient out $\K[x_1,\dots,x_n,y_\ii\text{'s}]/H^\vee$ by all the differences $y_\ii-1$ with $\ii\notin U$. These differences form a regular sequence by the previous lemma, so the quotient still is a Cohen--Macaulay ring. 
For each $\ii\in U$, let $d_\ii$ be a positive integer, a multiplicity in the sense of Definition~\ref{edgemultiplicity} for the $(d-1)$-edge $x^\ii$, and replace $y_\ii$ by the product
$$y_{\ii,1}y_{\ii,2}\cdots y_{\ii,d_\ii},$$
where the $y_{\ii,t}$'s are new variables. In this way we get from $H^\vee$ the ideal
$$(x_1,\dots,x_n)^{n-d+2}_{\sqf}
+\Big(\frac{x_1\cdots x_n}{x^\ii}y_{\ii,1}y_{\ii,2}\cdots y_{\ii,d_\ii}
\mid\ii\in U\Big)
+\Big(\frac{x_1\cdots x_n}{x^\ii}
\mid\ii\notin U\Big)$$
and the quotient by this ideal is still Cohen--Macaualy. (By the way, notice that the generators in the third summand---if there are any---, being of degree~$n-d+1$, divide some generators in the first summand, which have degree~$n-d+2$.)
   The Alexander dual of this ideal is
$$K:=(x_1,\dots,x_n)^d_{\sqf}+\Big(x^\ii y_{\ii,t} \mid t=1,\dots,d_\ii,\ii\in\binom{[n]}{d-1}\Big).$$
Then, again by Theorem~\ref{eigonreiner}, the ideal $K$  has a linear resolution.

\subsubsection{End of the proof}

See Section\til\ref{backgroundhypgraphs} for definitions and notations related to hypergraphs.

Take a $d$-uniform hypergraph $H$ on $[n]$, that is, with edge set $E\subset \binom{[n]}d$. Consider the ideal $\LIN(I)$ with $I=I_H\subset \K[x_1,\dots,x_n]$ the edge ideal of\til$H$, so that
$$\LIN(I)=\Big(x^\jj\mid\jj\in \binom{[n]}d\Big)+\Big(\frac{x^\jj}{x_j}y_\jj\mid\jj\in E,j\in\jj\Big).$$
For each $\jj\in E$ we may separate the variable $y_\jj$ and get monomials $\frac{x^\jj}{x_j}y_{\jj,j}$, one for each $j\in\jj$. This gives an ideal $K'$ which is a separation of $\LIN(I)$. For each $\ii$ of cardinality $d-1$ we have monomials $x^\ii\cdot y_{\ii\cup\{j\},j}$ where each $\ii\cup\{j\}$ is in $E$. The ideal $K'$ here constructed identifies with the ideal $K$ above, where $d_\ii$ is the number of sets $\ii\cup\{j\}$ contained in $E$. 
Hence $K'$, and so $\LIN(I)$, has linear resolution, and by the construction above its  Betti numbers only depend on the multiplicities $d_\ii$.

%%%%%%%%%%%%%%%%%%%%%%%%%%%%%%%%%%%

\subsection{Combinatorial interpretations by means of hypergraphs}\label{hyperlin}

We already mentioned the connection between commutative algebra and combinatorics provided by the Stanley--Reisner correspondence.
 Now we change point of view and for the rest of the section we focus on another way of linking combinatorics and commutative algebra: instead of simplicial complexes, we talk about hypergraphs.

\subsubsection{Background on hypergraphs}\label{backgroundhypgraphs}

A \tbs{hypergraph} is a pair $H=(V,\EE)$ with $V$ a finite set and $\EE\se \binom V2\setminus\{\emse\}$ a set of non-empty subsets of $V$. We call $V$ the set of \tbs{vertices} and denote it also by $V(H)$ and we call $\EE$ the set of \tbs{edges} (regardless of the cardinality) and we denote it also by $\EE(H)$. 
We say that a hypergraph $(V,\EE)$ is \tbs{$d$-uniform} if $\EE\se\binom Vd$, that is, if all the edges have the same cardinality $d$. Two vertices $v\ne w$ in $H$ are \tbs{neighbours} if there is an edge $E$ such that $v,w\in E$. For any vertex $v$, the \tbs{neighbourhood} of $v$ is
$$N(v):=\{w\in X\mid w \text{ is a neighbour of $v$}\}.$$

\begin{notation}
If $H=(V,\mathcal E)$ is a hypergraph and $W\se V$ is a subset, then the \tbs{induced hypergraph} on $W$, denoted $H_W$, is the subhypergraph of $H$ whose edge set is $\{E\in\mathcal E\mid E\se W\}$.
\end{notation}

\subsubsection{Hypergraphs with linear resolution}\label{hyplinres}

To a hypergraph $H=(V,\EE)$ with $\#V=n$, we associate a squarefree monomial ideal in $\K[x_1,\dots,x_n]$ called the \tbs{edge ideal} of $H$,
$$I_H:=\Big(\prod_{i\in E}x_i\mid E\in \EE\Big).$$
This provides a bijection between hypergraps with $V=\{1,\dots,n\}$ and squarefree monomial ideals in $\K[x_1,\dots,x_n]$. In their paper~\cite{HV}, Huy Tai H\`a and Adam Van Tuyl study the graded Betti numbers of~$I_H$ and give a characterization, under certain assumptions on the hypergraph~$H$, of the ideals~$I_H$ that have a linear resolution. A special case of this characterization is Fr\"oberg's Theorem~\ref{frofro}. The following definitions and results are taken from Sections~2, 4 and~5 of~\cite{HV}. We skip some assumptions which in our case are automatically satisfied because all our hypergraphs are $d$-uniform.

A \tbs{chain of length $n$} in $H$ is a finite sequence 
$$(E_0,v_1,E_1,v_2,E_2,v_3,\dots,E_{n-1},v_n,E_n)$$ where
\begin{enumerate}
\item[(1)] $v_1,\dots,v_n$ are pairwise distinct vertices of $H$;
\item[(2)] $E_0,\dots,E_n$ are pairwise distinct edges of $H$;
\item[(3)] $v_k\in E_{k-1}\cap E_k$  for all $k\in\{1,\dots,n\}$.
\end{enumerate}
Sometimes such a chain is denoted by $(E_0,\dots,E_n)$. 
 If $E$ and $E'$ are two edges, then $E$ and $E'$ are \tbs{connected} if, for some $n\in\N$, there is a chain $(E_0,\dots,E_n)$ where $E=E_0$ and $E'=E_n$. If $H$ is $d$-uniform, the chain connecting $E$ to $E'$ is a \tbs{proper chain} if $|E_{k-1}\cap E_k|=d-1$ for all $k\in\{1,\dots,n\}$. The (proper) chain is a \tbs{(proper) irredundant chain} of length $n$ if no proper subsequence is a (proper) chain from $E$ to $E'$. We define the \tbs{distance} between $E$ and $E'$ to be
$$\dist_H(E,E'):=\min\{n\mid (E=E_0,\dots,E_n=E) \text{ is a proper irredundant chain}\}.$$
If no proper irredundant chain exists, we set $\dist_H(E,E'):=\infty$.
A $d$-uniform hypergraph $H$ is said to be \tbs{properly-connected} if for any two edges $E$ and $E'$ with the property that $E\cap E'\ne\emse$, one has
$$\dist_H(E,E')=d-|E\cap E'|.$$
A $d$-uniform properly-connected hypergraph $H=(V,\mathcal E)$ is said to be \tbs{triangulated} if for every non-empty subset $W\se V$, the induced subhypergraph~$H_W$ contains a vertex $v\in W$ such that the induced hypergraph of $H_W$ on $N(v)\cup\{v\}$ is the $d$-complete hypergraph of order $|N(v)|+1$. The \tbs{edge diameter} of a $d$-uniform properly-connected hypergraph $H$ is
$$\diam(H):=\max\{\dist_H(E,E')\mid E,E'\in \EE\},$$
where the diameter is infinite if there exist two edges not connected by any proper chain.

Among several results concerning Betti numbers and resolutions, Huy Tai H\`a and Adam Van Tuyl proved in particular the following characterization.

\begin{lemma}[Corollary~7.6 of~\cite{HV}]\label{lemminoHV}
Let $H$ be a $d$-uniform, properly-connected, triangulated hypergraph. Then the following are equivalent:
\begin{itemize}
\item $I_H$ has a linear resolution;
\item $I_H$ as linear first syzygies;
\item $\diam(H)\le d$.
\end{itemize}
\end{lemma}

\begin{remark}
It's easy to see that the hypergraphs associated to  $\LIN(I)$ and $\Lin(I)$ are properly-connected. They are also triangulated: given a subset of vertices $W$, if $W$ contains no $y$-variables, then pick any $x_i$ as the vertex $v$ in the definition of triangulated hypergraph. If $W$ contains some $y_j$, then pick any of those as~$v$. So then the hypergraphs associated to $\LIN(I)$ and $\Lin(I)$ satisfy the assumptions in Lemma~\ref{lemminoHV}. Therefore, showing that the diameter is at most~$d$ or showing that $\LIN(I)$ and $\Lin(I)$ have linear first syzygies are alternative ways to prove that $\LIN(I)$ and $\Lin(I)$ have a linear resolution.
\end{remark}

%%%%%%%%%%%%%%%%%%%%%%%%%%%%%%%%%%%
%%%%%%%%%%%%%%%%%%%%%%%%%%%%%%%%%%%

\section{Linearization for all monomial ideals}\label{linnonequig}

In the previous sections the linearization construction has been defined only for monomial ideals which are \emph{equigenerated}, that is, generated in a single degree. In order to generalize the construction to arbitrary monomial ideals, we first introduce what we call \emph{equification} of a monomial ideal. This construction takes an arbitrary monomial ideal and returns an equigenerated monomial ideal. It bears some resemblance with the well-known \emph{homogenization} construction---which gives a homogeneous ideal starting from an arbitrary ideal---and this resemblance is made exact in Remark~\ref{equihomosimilarity}.

As a side note, observe that the words ``equification'' and ``equify'' already exist in English, as technical terms in trading and economics. In this paper there is no relation at all to those meanings. The word ``equification'' was suggested to me in analogy to ``sheafification'', which is a well-known process to make a presheaf into a sheaf.

\subsection{Equification of a monomial ideal}\label{omogeneizzmonomials}

Let us assume that $\K$ is a field of characteristic $0$.

\begin{definition}\label{defequification}
Let $I$ be a monomial ideal in $S=\K[x_1,\dots,x_n]$, with minimal system of monomial generators $G(I)=\{f_1,\dots,f_m\}$. Denote $d_j:=\deg(f_j)$ for all $j$ and $d:=\max\{d_j\mid j=1,\dots,m\}$. We define the \tbs{equification of $I$}as
$$I^\eq:=(f_1z^{d-d_1},f_2z^{d-d_2},\dots,f_mz^{d-d_m})$$
in the polynomial ring $S[z]=\K[x_1,\dots,x_n,z]$ with one extra variable $z$.
\end{definition}

\begin{lemma}\label{gensofequif}
The generators of $I^\eq$ in the definition are minimal.
\end{lemma}

\begin{proof}
Assume that $f_iz^{d-d_i}$ divides $f_jz^{d-d_j}$. Then $f_i$ divides $f_j$ and  $d-d_i\le d-d_j$. Since $f_iz^{d-d_i}$ and $f_jz^{d-d_j}$ have the same degree, we have $i=j$.
\end{proof}

\begin{remark}[recovering $I$ from $I^\eq$]\label{fromIeqtoI}
By setting $z=1$, we recover $I$. Of course this requires $z$ to be somehow distinguishable from the rest of the variables. Consider for instance $I=(x^3,y^2)$ in the polynomial ring~$\K[x,y]$. Then $I^\eq=(x^3,y^2z)$ in $\K[x,y,z]$. Clearly $x$ could not possibly be the ``equifying'' variable, but $y$ could be, so actually the same ideal $I^\eq$ could be obtained by equifying the ideal $(x^3,z)\se\K[x,z]$ by introducing a new variable~$y$. For this reason in this section we always assume $z$ to be a well-distinguised variable.
\end{remark}

\begin{remark}\label{equihomosimilarity}
The equification construction is somewhat similar to that of \emph{homogenization} (see for instance Section~3.2.1 of~\cite{HH}): for a polynomial $g\in S=\K[x_1,\dots,x_n]$ we may write uniquely $g=g_0+g_1+\dots+g_d$, where each $g_j$ is homogeneous of degree $j$ and $g_d\ne0$. Then the homogenization of $g$ is the polynomial
$$g^\hom:=\sum_{j=0}^dg_jz^{d-j}\in S[z].$$ The similarity with the equification is made precise by taking the elements in the support: if we have $G(I)=\{f_1,\dots,f_m\},$ then
$$I^\eq=\big(u\mid u\in\Supp\big((f_1+\dots+f_m)^\hom\big)\big).$$Here by \emph{support} of a polynomial we mean the set of its monomials with non-zero coefficient.
\end{remark}

%\begin{remark}
%How about generalizing this construction for, say, arbitrary homogeneous ideals? In that case the degree of the generators should be invariant, in a minimal system of homogeneous generators. 
%\end{remark}

\begin{remark}
Unlike the linearization, which is a functor as observed in Section~\ref{sectionwithfunctor}, the equification is not functorial, that is, it doesn't preserve inclusions. Take for instance, inside $\K[x,y]$, the ideals
$$I=(x^2,xy^3)\quad\se\quad J=(x^2,xy^2,y^5).$$
For their equifications
$$I^\eq=(x^2z^2,xy^3)\quad\text{and}\quad J^\eq=(x^2z^3,xy^2z^2,y^5)$$
we have $I^\eq\not\se J^\eq$ and $J^\eq\not\se I^\eq$.
\end{remark}

\begin{remark}
Recall that a monomial ideal is prime if and only if it is generated by a bunch of variables. And also recall that a monomial ideal is radical if and only if it is squarefree.
Then we have the following two equivalences:
$$\text{$I^\eq$ is prime}\quad\Leftrightarrow\quad\text{$I$ is prime}$$
(in which case $I^\eq=I$) and, perhaps more interestingly,
\begin{align*}
\text{$I^\eq$ is radical}\quad&\Leftrightarrow\quad \text{$I$ is radical and generated in}\\
&\qquad\quad\!\!\text{at most two adjacent degrees.}
\end{align*}
They are both clear. The second equivalence just means this: the original ideal can either  be generated in a single degree, and in this case $I^\eq=I$, or it can have generators of two distinct degrees, but they have to be in adjacent degrees so that $z$ appears with at most exponent~$1$ in the generators of $I^\eq$.
\end{remark}

\begin{remark}\label{eqintred}
A way to illustrate pictorially what happens with $I^\eq$ is as follows. This works for $n=2$. Think of the monomials in $\K[x,y]$ as lattice points in the plane with axes $x$ and $y$. For all $d$, consider the line $x+y=d$, which goes through all monomials of degree $d$ in $x$ and $y$. Then, what $(\cdot)^\eq$ does is that we add a new axis $z$ that comes out of the plane, we take the generators of $I$ of degree $d'$ (which are the ones lying on the line $x+y=d'$) and we bring them up to level $d-d'$. So, in particular, the ones of degree $d$ stay on the plane. See Figure~\ref{equifdrawing} for an example of this.
\begin{figure}
\begin{center}
\begin{tikzpicture} [>=latex,scale=.7]
%\draw [help lines] (0,0) grid (7,9);
% original graph
% vertices
\fill (.5,1) circle (0.12);
\coordinate [label= left: $x^3$] (x1) at (.5,1);
\fill (6,2) circle (0.12);
\coordinate [label=above : $y^4$] (x1) at (6,2);
\fill (2.5,2.5) circle (0.12);
\coordinate [label=below : $xy$] (x1) at (2.5,2.5);
% arrows
\draw [thick,->] (2,4) -- (-.5,-1);
\coordinate [label=right: $x$] (x1) at (-.5,-1);
\draw [thick,->] (2,4) -- (7,1.5);
\coordinate [label=right: $y$] (x1) at (7,1.5);
\draw [thick,->] (2,4) -- (2,9);
\coordinate [label=right: $z$] (x1) at (2,9);
% dashed
\draw [thick,dashed,gray] (1.5,3)--(2.5,2.5);
\draw [thick,dashed,gray] (3,3.5)--(2.5,2.5);
% parallel lines
\draw [thick,dashed,gray] (0,0)--(6,2);
\draw [thick,dashed,gray] (.5,1)--(5,2.5);
\draw [thick,dashed,gray] (1,2)--(4,3);
% plane
%\draw [thick,dashed,gray] (0,0)--(6,2)--(2,8)--(0,0);
\end{tikzpicture}\qquad
\begin{tikzpicture} [>=latex,scale=.7]
%\draw [help lines] (0,0) grid (7,9);
% original graph
% vertices
\fill [gray] (.5,1) circle (0.1);
%\coordinate [label=above left: $x^3$] (x1) at (.5,1);
\fill  (6,2) circle (0.12);
\coordinate [label=above right: $y^4$] (x1) at (6,2);
\fill [gray] (2.5,2.5) circle (0.1);
%\coordinate [label=below right: $xy$] (x1) at (2.5,2.5);
% new vertices
\fill  (.5,2) circle (0.12);
\coordinate [label=above left: $x^3z$] (x1) at (.5,2);
\fill  (2.5,4.5) circle (0.12);
\coordinate [label=below right: $xyz^2$] (x1) at (2.5,4.7);
% projections
\draw [thick,dashed,gray] (2.5,4.5)--(2.5,2.5);
\draw [thick,dashed,gray] (.5,2)--(.5,1);
% arrows
\draw [thick,->] (2,4) -- (-.5,-1);
\coordinate [label=right: $x$] (x1) at (-.5,-1);
\draw [thick,->] (2,4) -- (7,1.5);
\coordinate [label=right: $y$] (x1) at (7,1.5);
\draw [thick,->] (2,4) -- (2,9);
\coordinate [label=right: $z$] (x1) at (2,9);
% dashed
\draw [thick,dashed,gray] (1.5,3)--(2.5,2.5);
\draw [thick,dashed,gray] (3,3.5)--(2.5,2.5);
% parallel lines
%\draw [thick,dashed,gray] (0,0)--(6,2);
%\draw [thick,dashed,gray] (.5,1)--(5,2.5);
%\draw [thick,dashed,gray] (1,2)--(4,3);
% plane
\draw [thick,dashed,gray, fill=green, opacity=0.22] (0,0)--(6,2)--(2,8)--(0,0);
%\fill (0,0)--(6,2)--(2,8);
\end{tikzpicture}
\caption{We draw $I=(x^3,xy,y^4)$ on the left and $I^\eq=(x^3z,xyz^2,y^4)$ on the right. With a slight abuse of notation, the generator $x^3$ lies on the line $x+y=3$, $xy$ on the line $x+y=2$ and $y^4$ on the line $x+y=4$, all of them on the plane $z=0$. Those three parallel lines are all dashed, in the left picture. The generators of $I^\eq$ all lie on the plane $x+y+z=4$, in green. See Remark~\ref{eqintred}.}
\label{equifdrawing}
\end{center}
\end{figure}
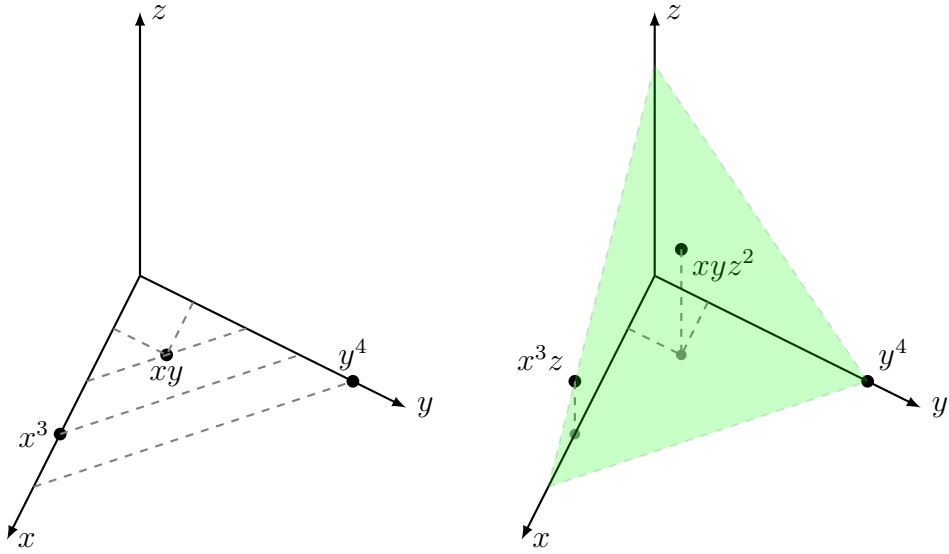
\end{remark}

\subsubsection{Betti numbers of $I^\eq$}

Unfortunately we are not able to give a complete, satisfactory description of the homological invariants of $I^\eq$, but we do provide some partial results: Proposition~\ref{inequalitybnumbs} and Proposition~\ref{redsyzeq}. Before stating the results in the end of the section, we discuss some examples and issues.

\begin{example}\label{exbettinumbsok}
Consider, respectively in $S:=\K[x_1,x_2,x_3]$ and in $T:=S[z]$, the ideals
$$I=(x_1^2,\ x_1x_2^2x_3^2,\ x_2^3x_3^2),\qquad
I^\eq=(x_1^2z^3,\ x_1x_2^2x_3^2,\ x_2^3x_3^2).$$
Then one has
$$\beta^S(I)=\begin{array}{c|cc}
&0&1\\
\hline
2&1&-\\
3&-&-\\
4&-&-\\
5&2&2\\
\hline
&3&2
\end{array}\qquad\text{and}\qquad
\beta^T(I^\eq)=\begin{array}{c|cc}
&0&1\\
\hline
5&3&1\\
6&-&-\\
7&-&-\\
8&-&1\\
\hline
&3&2.
\end{array}$$
and the minimal resolutions look like
$$0\to S^2\xrightarrow{\left(\begin{array}{cc}
0&x_2^2x_3^2\\
x_2&-x_1\\
-x_1&0
\end{array}\right)}
S^3\xrightarrow{(x_1^2,\ x_1x_2^2x_3^2,\ x_2^3x_3^2)}I\to0$$
and
$$0\to T^2\xrightarrow{\left(\begin{array}{cc}
0&x_2^2x_3^2\\
x_2&-x_1z^3\\
-x_1&0
\end{array}\right)}
T^3\xrightarrow{(x_1^2z^3,\ x_1x_2^2x_3^2,\ x_2^3x_3^2)}
I^\eq\to0.$$
Notice that in this example the total Betti numbers of $I$ and $I^\eq$ are equal. Of course the grading cannot be the same, because the very first column in the Betti table records the degrees of the generators, and the generators of $I^\eq$ all have the same degree, which is the highest degree of the minimal generators of $I$. But at least the $0$-th total Betti number of $I$ and that of $I^\eq$ will be the same, due to Lemma~\ref{gensofequif}. We also notice, in this example, that the maps in the two resolutions are quite similar. They only differ for the presence of some powers of~$z$. Unfortunately this is not the case in general, as Example~\ref{esempiobruttesizigie} will show.
\end{example}

%\begin{example}[maybe too symmetric]
%Consider the following, in $k[x_1,\dots,x_3]$ and $k[x_1,\dots,x_n,z]$:
%\begin{align*}
%I&=(x_1x_2^2,\ x_1^2x_2,\ x_2^3x_3^2,\ x_2^2x_3^3),\\
%I^\eq&=(x_1x_2^2z^2,\ x_1^2x_2z^2,\ x_2^3x_3^2,\ x_2^2x_3^3).
%\end{align*}
%The Betti tables are
%$$\beta(S/I)=\begin{array}{c|cccc}
%&0&1&2&3\\
%\hline
%0&1&-&-&-\\
%1&-&-&-&-\\
%2&-&2&1&-\\
%3&-&-&-&-\\
%4&-&2&3&1
%\end{array}\qquad
%\beta(S/I^\eq)=\begin{array}{c|cccc}
%&0&1&2&3\\
%\hline
%0&1&-&-&-\\
%1&-&-&-&-\\
%2&-&-&-&-\\
%3&-&-&-&-\\
%4&-&4&2&-\\
%5&-&-&-&-\\
%6&-&-&2&1
%\end{array}$$
%Again, when taking $I^\eq$, the matrices in the resolution are almost the same as the ones for $I$, except for the presence of some powers of $z$ here and there.
%\end{example}

\begin{notation} 
Let $I$ be a monomial ideal in $S=\K[x_1,\dots,x_n]$, with $G(I)=\{f_1,\dots,f_m\}$. Setting $d_j:=\deg(f_j)$ for all $j$, we denote $d:=\max\{d_j\mid j=1,\dots,m\}$ and $\delta:=\min\{d_j\mid j=1,\dots,m\}$. 
%Let  now
%$$0\to F_p\stackrel{\vfi_p}\la\dots \stackrel{\vfi_2}\la F_1\stackrel{\vfi_1}\la F_0\stackrel\vep\la I\to 0$$
%be a minimal graded $S$-free resolution of $I$. We would like to obtain  a minimal graded $T$-free resolution of $I^\eq$, let's call it
%$$0\to G_t\stackrel{\psi_p}\la\dots\stackrel{\psi_2}\la G_1\stackrel{\psi_1}\la G_0\stackrel\zeta\la I^\eq\to 0.$$
\end{notation}

\begin{remark}
For what concerns the first step of a minimal graded free resolution of $I$, we have
$$\vep\:\bigoplus_{j=1}^mS(-d_j)\la I,\qquad e_j\map f_j,$$
and the corresponding map for $I^\eq$ is easily given as
$$\zeta\:\bigoplus_{j=1}^mT(-d)\la I^\eq,\qquad \eta_j\map f_jz^{d-d_j}.$$ Let's continue by considering the sygyzy module $\Syz(I):=\ker(\vep)$. A syzygy of~$I$ is $(p_1,p_2,\dots,p_m)\in S^m$ such that $p_1f_1+p_2f_2+\dots+p_mf_m=0$. Hence, by writing $d-\delta=d_j-\delta+d-d_j$, we get
\begin{align*}
0&=(p_1f_1+p_2f_2+\dots+p_mf_m)z^{d-\delta}\\
&=p_1z^{d_1-\delta}(f_1z^{d-d_1})+p_2z^{d_2-\delta}(f_2z^{d-d_2})+\dots+p_mz^{d_m-\delta}(f_mz^{d-d_m}),
\end{align*}
so that $(p_1z^{d_1-\delta},p_2z^{d_2-\delta},\dots,p_mz^{d_m-\delta})$ is a syzygy for $I^\eq$. The problem is that the syzygies of $I^\eq$ obtained in this way do not generate in general all of $\Syz(I^\eq)$, as showed in Example~\ref{esempiobruttesizigie}.
\end{remark}

\begin{notation}\label{redtrivsyz}
Denote by $e_1,\dots,e_m$ the standard basis of $S^m$. For $I$ with $G(I)=\{f_1,\dots,f_m\}$, we have syzygies $f_je_i-f_ie_j$. These clearly map to zero, but they can be refined as
\begin{align*}
\sigma_{ij}&:=\frac{f_j}{\gcd(f_i,f_j)}e_i-\frac{f_i}{\gcd(f_i,f_j)}e_j\\
&=\frac{\lcm(f_i,f_j)}{f_i}e_i-\frac{\lcm(f_i,f_j)}{f_j}e_j,
\end{align*}
which we call the \tbs{reduced trivial syzygies} of $I$. A well-known theorem by Schreyer states that the reduced trivial syzygies generate all of $\Syz(I)$. (See Theorem~15.10 of~\cite{EisCA}.)
\end{notation}

\begin{example}\label{esempiobruttesizigie}
Consider now $S=\K[x_1,\dots,x_4]$, $T=S[z]$, and
$$I=(x_1x_2x_4,\ x_1^2x_2^2x_3,\ x_3^3x_4^3),\qquad 
I^\eq=(x_1x_2x_4z^3,\ x_1^2x_2^2x_3z,\ x_3^3x_4^3).$$ Then we get the Betti tables
$$\beta^S(I)=\begin{array}{c|cc}
&0&1\\
\hline
3&1&-\\
4&-&-\\
5&1&1\\
6&1&-\\
7&-&1\\
\hline
&3&2
\end{array}\qquad\text{and}\qquad
\beta^T(I^\eq)=\begin{array}{c|ccc}
&0&1&2\\
\hline
6&3&-&-\\
7&-&-&-\\
8&-&1&-\\
9&-&-&-\\
10&-&2&-\\
11&-&-&1\\
\hline
&3&3&1.
\end{array}$$
As already remarked, we always have $\beta^S_0(I)=\beta^T_0(I^\eq)$, but in this example we have strict inequalities $\beta^S_i(I)<\beta^T_i(I^\eq)$ for $i=1,2$. The matrices correspoding to the first syzygies of $I$ and $I^\eq$ are respectively
$$\left(\begin{array}{cc}
x_1x_2x_3&x_3^3x_4^2\\
-x_4&0\\
0&-x_1x_2
\end{array}\right)\quad\text{and}\quad
\left(\begin{array}{ccc}
x_1x_2x_3&x_3^3x_4^2&0\\
-x_4z^2&0&x_3^2x_4^3\\
0&-x_1x_2z^3&-x_1^2x_2^2z
\end{array}\right).$$
For $I$ the reduced trivial syzygy $\sigma_{23}$ is redundant, because we have
$\sigma_{23}=-x_3^2x_4^2\sigma_{12}+x_1x_2\sigma_{13}$. But for $I^\eq$ the corresponding syzygy, showing as third column in the matrix, is not redundant.
\end{example}

The following lemma is a classical result. We include a proof of it for sake of completeness. 

\begin{lemma}\label{lemmettogenerale}
Let $R$ be a polynomial over a field and let $M$ be a finitely generated $R$-module. Let $a\in R$ be $R$-regular and $M$-regular. Let $\F$ be a free resolution of $M$ over $R$. Then $\F\otimes_RR/(a)$ is a free resolution of $M/aM$ over $R/(a)$.
\end{lemma}

\begin{proof}
We may write the modules in the resolution $\F$ as  $F_i=R^{n_i}$ for some $n_i\in\N$. By the distributive property of the tensor product we have
$$\bigg(\bigoplus_{j=1}^{n_i}R\bigg)\otimes_RR/(a)=\bigoplus_{j=1}^{n_i}(R\otimes_RR/(a))=\bigoplus_{j=1}^{n_i}R/(a).$$
Moreover, since $-\otimes_RR/(a)$ is a functor, $\F\otimes_RR/(a)$ is a complex. To conclude, we show that it is exact. We have $H_i(\F\otimes_RR/(a))\iso\Tor^R_i(M,R/(a))$ and we use the commutativity of~$\Tor$. That is, by the $R$-regularity of $a$, we know that
$$\G:\quad 0\la R\stackrel a\la R$$
is a free resolution of $R/(a)$ over $R$, so that $\Tor^R_i(M,R/(a))\iso H_i(M\otimes_R\G)$. But this homology is~$0$ for $i>0$ because
$$M\otimes_R\G:\quad 0\la M\stackrel a\la M$$ is exact due to the $M$-regularity of $a$.
\end{proof}

\begin{proposition}\label{inequalitybnumbs}
We have $\beta_0^S(I)=\beta_0^T(I^\eq)$ and $\beta^S_i(I)\le\beta^T_i(I^\eq)$ for all $i>0$.
\end{proposition}

\begin{proof}
The equality $\beta_0^S(I)=\beta_0^T(I^\eq)$ is the content of Lemma~\ref{gensofequif}. For the rest of the proof, we specify Lemma~\ref{lemmettogenerale} to our setting:
$$R=T:=\K[x_1,\dots,x_n,z],\qquad M=I^\eq,\qquad a=z-1.$$
We observe that $z-1$ is clearly $T$-regular and hence $I^\eq$-regular. Moreover, $T/(z-1)=S$ and $I^\eq/(z-1)I^\eq=I$. So,  let now 
$$\F:\quad 0\la T^{\beta_p}\la T^{\beta_{p-1}}\la\dots\la T^{\beta_0}$$
be the \emph{minimal} graded free resolution of $I^\eq$, where $\beta_i=\beta_i^T(I^\eq)$. Then we get from Lemma~\ref{lemmettogenerale} we get that
$$\F\otimes_TT/(z-1):\quad 0\la S^{\beta_p}\la S^{\beta_{p-1}}\la\dots\la S^{\beta_0}$$
is a free resolution of $I$, possibly not minimal. A well-known result (see for instance Theorem~7.5 of~\cite{Pe}) states that any resolution contains the minimal one as a direct summand, and therefore we get the desired inequalities.
\end{proof}

Recall now the notation for the reduced trivial syzygies in Notation~\ref{redtrivsyz} and recall that the reduced trivial syzygies generate $\Syz(I)$, but some of them might be redundant.

\begin{lemma}\label{minimalrts}
The reduced trivial syzygy~$\sigma_{ij}$ is redundant if and only if there exists $k\notin\{i,j\}$ such that  $\lcm(f_k,f_i)$ and $\lcm(f_k,f_j)$ divide $\lcm(f_i,f_j)$.
\end{lemma}

\begin{proof}
$(\Leftarrow)$ This is a special case, with a different notation, of Proposition~8 in Section~2.9 of~\cite{CLO}. We inlcude a short proof for sake of completeness: by assumption, we have the monomials $u:=\lcm(f_i,f_j)/\lcm(f_k,f_i)$ and $v:=\lcm(f_i,f_j)/\lcm(f_k,f_j)$, so that
\begin{align*}
-u\sigma_{ki}+v\sigma_{kj}&=-\frac{\lcm(f_i,f_j)}{\lcm(f_k,f_i)}\Big(\frac{\lcm(f_k,f_i)}{f_k}e_k-\frac{\lcm(f_k,f_i)}{f_i}e_i\Big)\\
&\qquad+\frac{\lcm(f_k,f_j)}{\lcm(f_k,f_j)}\Big(\frac{\lcm(f_k,f_j)}{f_k}e_k-\frac{\lcm(f_k,f_j)}{f_j}e_j\Big)\\
&=\frac{\lcm(f_i,f_j)}{f_i}e_i-\frac{\lcm(f_i,f_j)}{f_j}e_j=\sigma_{ij}.
\end{align*}

$(\Rightarrow)$ Assuming that $\sigma_{ij}$ is redundant, we may write
$$\sigma_{ij}=\sum_{\{k,\ell\}\ne\{i,j\}}p_{k\ell}\sigma_{k\ell}$$
for some $p_{k\ell}$. In particular the coefficients of $e_i$ and $e_j$ on both sides are equal, so that isolating all terms which involve $i$, and respectively $j$, we get  $$\frac{\lcm(f_i,f_j)}{f_i}=p_{i\ell_1}\frac{\lcm(f_i,f_{\ell_1})}{f_i}+\dots+p_{i\ell_t}\frac{\lcm(f_i,f_{\ell_t})}{f_i}$$
and a similar expression for $\lcm(f_i,f_j)/f_j$. If we multiply by $f_i$, and respectively by $f_j$, and observe that the ideals generated by the least common multiples on the right-hand side are monomial ideals, this means that there exist~$\ell$ and~$k$ such that $\lcm(f_i,f_\ell)$  and $\lcm(f_k,f_j)$ divide $\lcm(f_i,f_j)$. In particular $f_k$ divides $\lcm(f_i,f_j)$, so that indeed both $\lcm(f_k,f_j)$ and $\lcm(f_i,f_k)$ divide $\lcm(f_i,f_j)$.
\end{proof}

Denote
\begin{align*}
\sigma_{ij}^\eq&:=\frac{\lcm(g_i,g_j)}{g_i}e_i-\frac{\lcm(g_i,g_j)}{g_j}e_j\\
&\ =\frac{\lcm(f_i,f_j)z^{\max\{0,d_i-d_j\}}}{f_i}e_i
-\frac{\lcm(f_i,f_j)z^{\max\{0,d_j-d_i\}}}{f_j}e_j
\end{align*}
the reduced trivial syzygies for $I^\eq$, where $g_i=f_iz^{d-d_i}$.

\begin{proposition}\label{redsyzeq}
The reduced trivial syzygy~$\sigma_{ij}^\eq$ is redundant if and only if there exists $k\notin\{i,j\}$ such that  $\lcm(f_k,f_i)$ and $\lcm(f_k,f_j)$ divide $\lcm(f_i,f_j)$ and $\min\{d_i,d_j\}\le d_k$.
\end{proposition}

\begin{proof}
This follows by Lemma~\ref{minimalrts}: $\lcm(g_i,g_j)=\lcm(f_i,f_j)z^{\max\{d-d_i,d-d_j\}}$ is divided by $\lcm(g_i,g_k)$ and $\lcm(g_k,g_j)$ if and only if $\lcm(f_i,f_j)$ is divided by $\lcm(f_i,f_k)$ and $\lcm(f_k,f_j)$ and additionally $\max\{d-d_i,d-d_j\}\ge d-d_k$, which is equivalent to $\min\{d_i,d_j\}\le d_k$.
\end{proof}

\subsubsection{The $\lcm$-lattice of $I^\eq$}\label{lcmlattices}

We recall the notion of $\lcm$-lattice of a monomial ideal and we compare that of $I$ and of $I^\eq$. This tool provides a way to construct a resolution of $I^\eq$, albeit not necessarily minimal.
The background for this section is taken from Section~58 of~\cite{Pe}.  We start by recalling that a \tbs{lattice} is a poset in which every pair of elements has a least upper bound and a greatest lower bound.

\begin{definition}
Let $I$ be a monomial ideal with $G(I)=(f_1,\dots,f_m)$. The \tbs{lcm-lattice} of $I$ is the lattice that has as elements the least common multiples of the subsets of $\{f_1,\dots,f_m\}$, ordered by divisibility. We denote the $\lcm$-lattice of $I$ by $L_I$. 
\end{definition}

Notice in particular that the bottom element in~$L_I$ is~$1$, which is the least common multiple of the empty set. For more details on the notion of $\lcm$-lattice, see Section~58 of~\cite{Pe}.

\begin{example}\label{esempiobruttesizigiecontinua}
Consider the ideal $I$ in Example~\ref{esempiobruttesizigie}: 
$$I=(x_1x_2x_4,\ x_1^2x_2^2x_3,\ x_3^3x_4^3)\ \se\  S=\K[x_1,\dots,x_4].$$
The $\lcm$-lattice $L_I$ is depicted in Figure~\ref{lcmlatticefirstex}.
\begin{figure}
\begin{center}
\begin{tikzpicture} [>=latex, xscale=.85, yscale=.7]
%\draw [help lines] (0,0) grid (4,6);
% vertices
\fill (2,6) circle (0.1);
\coordinate [label=above: $x_1^2x_2^2x_3^3x_4^3$] (26) at (2,6);
\fill (0,4) circle (0.1);
\coordinate [label=left: $x_1^2x_2^2x_3x_4$] (04) at (0,4);
\fill (2,4) circle (0.1);
\coordinate [label=right: $x_1x_2x_3^3x_4^3$] (24) at (2,4);
%\fill (4,4) circle (0.1);
\fill (0,2) circle (0.1);
\coordinate [label=left: $x_1x_2x_4$] (02) at (0,2);
\fill (2,2) circle (0.1);
\coordinate [label= right: $x_1^2x_2^2x_3$] (BLG) at (2,2);
\fill (4,2) circle (0.1);
\coordinate [label=right: $x_3^3x_4^3$] (42) at (4,2);
\fill (2,0) circle (0.1);
\coordinate [label=below: $1$] (1) at (2,0);
% edges
\draw [thick] (2,0) -- (0,2)--(0,4)--(2,6)--(2,4)--(4,2)--(2,0)--(2,2)--(0,4);
\draw [thick] (0,2)--(2,4);
\end{tikzpicture}
\caption{The $\lcm$-lattice of $I=(x_1x_2x_4, x_1^2x_2^2x_3, x_3^3x_4^3)$ in Example~\ref{esempiobruttesizigiecontinua}.}
\label{lcmlatticefirstex}
\end{center}
\end{figure}
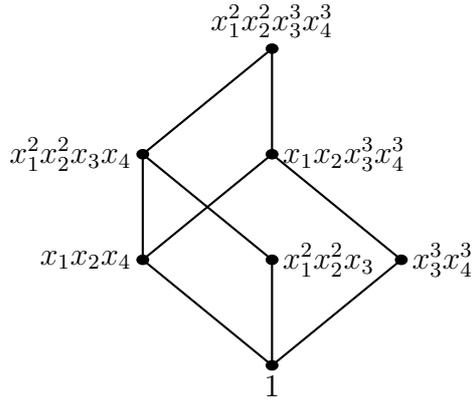
\end{example}

\begin{remark}\label{remarksublattice}
The $\lcm$-lattice $L_I$ of $I$ is isomorphic to a sublattice of the $\lcm$-lattice $L_{I^\eq}$ of $I^\eq$. To see this, we can set $z=1$ in $L_{I^\eq}$ and observe that if we multiply any two monomials $u$ and $v$ in the variables $x_1,\dots,x_n$ by some powers of~$z$, then we simply have
$$\lcm(uz^a,vz^b)=\lcm(u,v)z^{\max\{a,b\}}.$$
The difference, as illustrated in the next example, is that we might find some redundancies after setting~$z=1$.
\end{remark}

\begin{example}\label{alotoflcmlatt}
We compare the $\lcm$-lattices of $I$ and $I^\eq$: if $I$ is the ideal in Example~\ref{exbettinumbsok}, then $L_I\iso L_{I^\eq}$. The two lattices are drawn in  Figure~\ref{tuttobene}. On the other hand, if $I$ is the ideal in Example~\ref{esempiobruttesizigiecontinua}, then the isomorphic copy of~$L_I$ inside $L_{I^\eq}$ is strictly contained inside $L_{I^\eq}$. The two lattices are drawn in Figure~\ref{unsaccodilcm}. In $L_I$, the dashed part is a redundancy that we get  by setting $z=1$ in $L_{I^\eq}$; compare it with the drawing in Figure~\ref{lcmlatticefirstex}.
\begin{figure}
\begin{center}
\begin{tikzpicture} [>=latex, xscale=0.85, yscale=0.65]
%\draw [help lines] (0,8) grid (4,14);
%%%%%%%%%%%%%%%
%vertices
\fill (2,14) circle (0.1);
\coordinate [label=above: $x_1^2x_2^3x_3^2$] (214) at (2,14);
\fill (3,12) circle (0.1);
\coordinate [label=right: $x_1x_2^3x_3^2$] (312) at (3,12);
\fill (1,12) circle (0.1);
\coordinate [label=left: $x_1^2x_2^2x_3^2$] (112) at (1,12);
\fill (4,10) circle (0.1);
\coordinate [label=right: $x_2^3x_3^2$] (410) at (4,10);
\fill (2,10) circle (0.1);
\coordinate [label=right: $x_1x_2^2x_3^2$] (210) at (2,10);
\fill (0,10) circle (0.1);
\coordinate [label=left: $x_1^2$] (010) at (0,10);
\fill (2,8) circle (0.1);
\coordinate [label=below: $1$] (28) at (2,8);
% edges
\draw [thick] (2,8)--(0,10)--(1,12)--(2,14)--(3,12)--(4,10)--(2,8)--(2,10)--(3,12);
\draw [thick] (2,10)--(1,12);
%%%%%%%%%%%%%%%
\end{tikzpicture}\qquad\quad
\begin{tikzpicture} [>=latex, xscale=0.85, yscale=0.65]
%\draw [help lines] (0,8) grid (4,14);
%%%%%%%%%%%%%%%
%vertices
\fill (2,14) circle (0.1);
\coordinate [label=above: $x_1^2x_2^3x_3^2z^3$] (214) at (2,14);
\fill (3,12) circle (0.1);
\coordinate [label=right: $x_1x_2^3x_3^2$] (312) at (3,12);
\fill (1,12) circle (0.1);
\coordinate [label=left: $x_1^2x_2^2x_3^2z^3$] (112) at (1,12);
\fill (4,10) circle (0.1);
\coordinate [label=right: $x_2^3x_3^2$] (410) at (4,10);
\fill (2,10) circle (0.1);
\coordinate [label=right: $x_1x_2^2x_3^2$] (210) at (2,10);
\fill (0,10) circle (0.1);
\coordinate [label=left: $x_1^2z^3$] (010) at (0,10);
\fill (2,8) circle (0.1);
\coordinate [label=below: $1$] (28) at (2,8);
% edges
\draw [thick] (2,8)--(0,10)--(1,12)--(2,14)--(3,12)--(4,10)--(2,8)--(2,10)--(3,12);
\draw [thick] (2,10)--(1,12);
%%%%%%%%%%%%%%%
\end{tikzpicture}
\caption{Comparison of the $\lcm$-lattices of $I=(x_1^2, x_1x_2^2x_3^2, x_2^3x_3^2, )$,  on the left, and~$I^\eq$, on the right. In this case we have $L_I\iso L_{I^\eq}$. See  Example~\ref{alotoflcmlatt}. }
\label{tuttobene}
\end{center}
\end{figure}
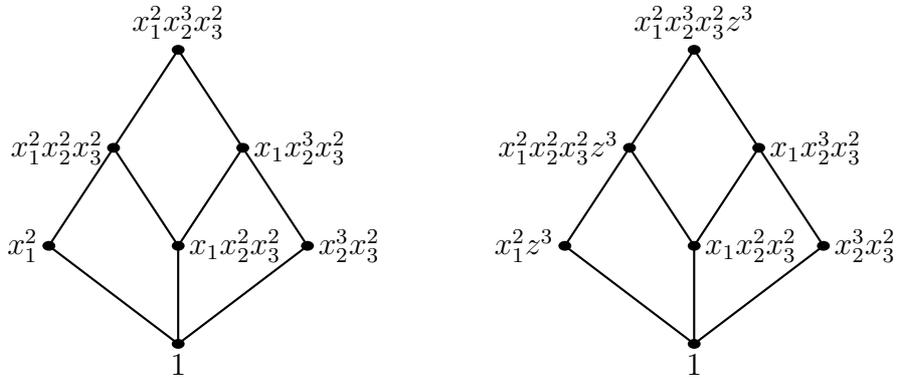

\begin{figure}
\begin{center}
\begin{tikzpicture} [>=latex, scale=0.75]
%\draw [help lines] (0,0) grid (5,6);
%%%%%%%%%%%%%%%
% vertices
\fill (2,6) circle (0.1);
\coordinate [label=above: $x_1^2x_2^2x_3^3x_4^3$] (26) at (2,6);
\fill (0,4) circle (0.1);
\coordinate [label=above left: $x_1^2x_2^2x_3x_4$] (04) at (0.3,4);
\fill (2,4) circle (0.1);
\coordinate [label= right: $x_1x_2x_3^3x_4^3$] (24) at (2,4);
\fill (0,2) circle (0.1);
\coordinate [label=below left: $x_1x_2x_4$] (02) at (0.2,2);
\fill (2,2) circle (0.1);
\coordinate [label= right: $x_1^2x_2^2x_3$] (BLG) at (2,2);
\fill (4.7,2) circle (0.1);
\coordinate [label= right: $x_3^3x_4^3$] (42) at (4.7,2);
\fill (2,0) circle (0.1);
\coordinate [label=below: $1$] (1) at (2,0);
% edges
\draw [thick] (2,0)--(0,2)--(0,4)--(2,6)--(2,4)--(4.7,2)--(2,0)--(2,2)--(0,4);
\draw [thick] (0,2)--(2,4);
% dashed part
\fill [gray] (4.7,4) circle (0.1);
\coordinate [label=above right: $x_1^2x_2^2x_3^3x_4^3$] (474) at (4.3,4);
\draw [thick,dashed,gray] (4.7,2)--(4.7,4)--(2,2);
\draw [thick,dashed,gray] (4.7,4)--(2,6);
\end{tikzpicture}
\begin{tikzpicture} [>=latex, scale=0.75]
%\draw [help lines] (0,0) grid (5,6);
%%%%%%%%%%%%%%%
% vertices
\fill (2,6) circle (0.1);
\coordinate [label=above: $x_1^2x_2^2x_3^3x_4^3z^3$] (26) at (2,6);
\fill (0,4) circle (0.1);
\coordinate [label=above left: $x_1^2x_2^2x_3x_4z^3$] (04) at (0.4,4);
\fill (2,4) circle (0.1);
\coordinate [label= above right: $x_1x_2x_3^3x_4^3z^3$] (24) at (1.9,3.8);
\fill (0,2) circle (0.1);
\coordinate [label=below left: $x_1x_2x_4z^3$] (02) at (0.3,2);
\fill (2,2) circle (0.1);
\coordinate [label= right: $x_1^2x_2^2x_3z$] (BLG) at (2,2);
\fill (5,2) circle (0.1);
\coordinate [label=right: $x_3^3x_4^3$] (42) at (5,2);
\fill (2,0) circle (0.1);
\coordinate [label=below: $1$] (1) at (2,0);
% edges
\draw [thick] (2,0)--(0,2)--(0,4)--(2,6)--(2,4)--(5,2)--(2,0)--(2,2)--(0,4);
\draw [thick] (0,2)--(2,4);
% dashed part
\fill  (5,4) circle (0.1);
\coordinate [label=right: $x_1^2x_2^2x_3^3x_4^3z$] (474) at (5,4);
\draw [thick] (5,2)--(5,4)--(2,2);
\draw [thick] (5,4)--(2,6);
\end{tikzpicture}
\caption{Comparison of the $\lcm$-lattices of $I=(x_1x_2x_4,x_1^2x_2^2x_3,x_3^3x_4^3)$, on the left, and $I^\eq$, on the right. See Example~\ref{alotoflcmlatt}. Compare the $\lcm$-lattice on the left with the one drawn in Figure~\ref{lcmlatticefirstex}.}
\label{unsaccodilcm}
\end{center}
\end{figure}
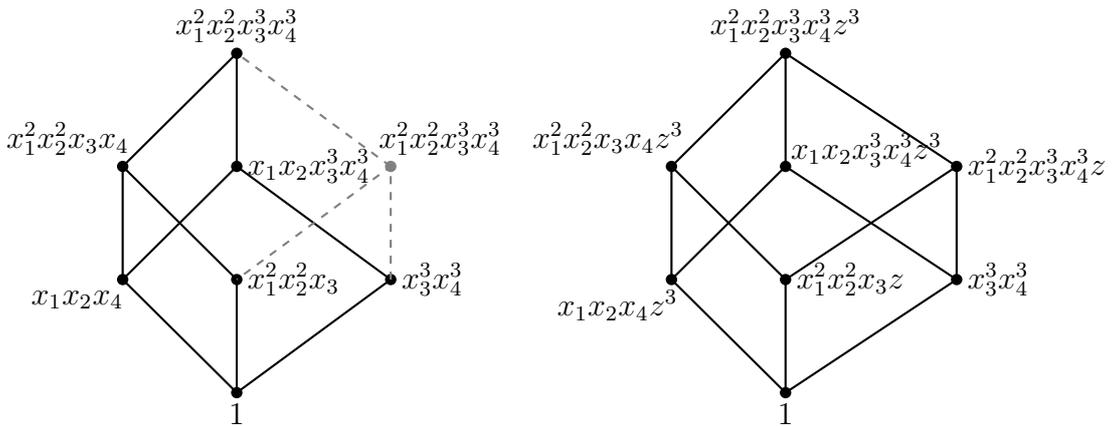
\end{example}

The material in this last part of the section is taken from Section~62 of~\cite{Pe}, which in turn follows~\cite{Lyub} and~\cite{Nov}.

\begin{definition}\label{defrootmap}
Given a monomial ideal $J$ and monomials $u_1,\dots,u_s$, a map $h\:L_J\setminus\{1\}\to\{u_1,\dots,u_s\}$ is called a \tbs{rooting map} for $J$ if the following conditions hold:
\begin{enumerate}
\item[(1)] for each $v\in L_J\setminus\{1\}$, $h(v)$ divides $v$;
\item[(2)] if $v,v'\in L_J\setminus\{1\}$ are such that $h(v)$ divides $v'$ and $v'$ divides $v$, then $h(v)=h(v')$.
\end{enumerate}
For each non-empty subset $U\se\{u_1,\dots,u_s\}$, set $h(U):=h(\lcm(u\mid u\in U))$. The subset $U$ is \tbs{unbroken} if $h(U)\in U$, and $U$ is \tbs{rooted} if all non-empty subsets of $U$ are unbroken. All rooted subsets of $\{u_1,\dots,u_s\}$, together with the empty set~$\emse$, form the \tbs{rooted complex} of $h$, which is denoted by $\RC_h$.
\end{definition}

\begin{proposition}[Novik, \cite{Nov}]\label{proprootcomplex}
If $J\se S$ is a monomial ideal and $h$ is a rooting map for $J$, then $\RC_h$ supports a simplicial free resolution of~$S/J$.
\end{proposition}

For a proof and additional information, we refer to Theorem 60.2 of~\cite{Pe}.

\begin{remark}[a resolution for $I^\eq$]
Let $I\se\K[x_1,\dots,x_n]$ be a monomial ideal and consider the $\lcm$-lattice of the equification~$I^\eq$. For a monomial $u\in\K[x_1,\dots,x_n,z]$, denote by $\cl u$ the corresponding monomial in $\K[x_1,\dots,x_n]$ where we set $z=1$. The map
$$h\:L_{I^\eq}\setminus\{1\}\la L_I\setminus\{1\},\quad u\map\cl u$$
is a rooting map for $I^\eq$. Indeed, condition~(1) in Definition~\ref{defrootmap} is satisfied because $\cl u$ divides $u$. As for condition~(2), assume that  $v$ and $v'$ in $L_{I^\eq}\setminus\{1\}$ are such that $h(v)$ divides $v'$ and $v'$ divides $v$. These last two things, for our specific $h$, imply that $h(v)$ divides $h(v')$ and $h(v')$ divides $h(v)$, so that $h(v)=h(v')$. So $h$ constructed here is a rooting map and we can apply Proposition~\ref{proprootcomplex} to~$h$.
\end{remark}

\subsection{Linearization of arbitrary monomial ideals}

In this section we give a generalization of the linearization construction that works for any monomial ideal, not necessarily equigenerated.

\begin{definition}\label{linpertutti}
The \tbs{linearization} of a monomial ideal $I$ is defined as 
$$\Lin(I):=\Lin(I^\eq),$$ where $\Lin$ on the right-hand side is the one introduced in Definition~\ref{deflinearization}, for equigenerated monomial ideals.
\end{definition}

Notice that, in case $I$ is equigenerated, then we get $I^\eq=I$ and the only difference is that we consider the ideal $\Lin(I)$ in a polynomial ring with one more variable. This does not affect~$\Lin$ in any sensible way, especially from the homological point of view. It does however interfere with $\LIN$, and therefore we focus only on $\Lin$ in this section. 

\begin{remark}
From $\Lin(I^\eq)$, as observed in Remark~\ref{tornarindietrosingologrado}, one can recover~$I^\eq$. And, as discussed in Remark~\ref{fromIeqtoI}, from $I^\eq\subset S[z]$ we can find $I\subset S$ simply by setting $z=1$.
\end{remark}

To conclude the section we examine one last matter: the presence of the variable~$z$.
 Since the  linearization in Definition~\ref{deflinearization} did not involve any variable~$z$, is it possible to define the linearization for any monomial ideal without going through the equification? The problem is that we don't really know how to deal with the complete part of $\Lin(I)$. For what concerns the last part, one possibility is just to define it as it is defined for the case of an equigenerated ideal. The problem is that this last part alone does not seem to have any nice properties. Not even in the equigenerated case, in fact, as discussed in Section~\ref{domandeaperte}.
Or, assume that we do construct $\Lin$ for arbitrary monomial ideals as in Definition~\ref{linpertutti}. In that case, in analogy to setting $z=1$ in order to get $I$ back from $I^\eq$, what happens to $\Lin(I^\eq)$ if we set $z=1$? The answer is discussed below.

\begin{notation}
Let $I\se S:=\K[x_1,\dots,x_n]$ be a monomial ideal with minimal system of monomial generators $G(I)=\{f_1,\dots,f_m\}$. Denote $d_j:=\deg(f_j)$ for all $j$ and also let $d:=\max\{d_j\mid j=1,\dots,m\}$ and $\delta:=\min\{d_j\mid j=1,\dots,m\}$. For each $i\in\{1,\dots,n\}$, let $M_i$ be the highest exponent of the variable $x_i$ occuring in $G(I)$. Denote $v:=(M_1,\dots,M_n)$ the vector of highest exponents. Lastly, for each $j\in\{1,\dots,m\}$, denote $g_j:=f_jz^{d-d_j}$ the generators of $I^\eq$.
\end{notation}

Assume that not all generators of $I$ are in degree $d$, otherwise the equification would be the ideal itself. Namely, assume that $\delta<d$. The vector of highest exponents for $I^\eq$ is
$$v^\eq=(M_1,\dots,M_n,d-\delta).$$
So we get
\begin{align*}
\Lin(I^\eq)&=(\text{monomials in $x_1,\dots,x_n,z$ of degree $d$, with vector below $v^\eq$})\\
&\quad+\big(\frac{g_jy_j}{x_k}\mid j=1,\dots,m, \text{where $x_k$ divides $g_j$}\big)\\
&\quad+\big(\frac{g_jy_j}z\mid j=1,\dots,m,\text{where $z$ divides $g_j$}\big).
\end{align*}
For the last part, we could decide not to treat $z$ as one of the ``normal'' variables $x_i$'s, and it would not make a difference if we are going to set~$z=1$ afterwards. That is, we could decide not to put the generators of the form $\frac{f_j}zy_j$ in the last part of $\Lin(I^\eq)$. Indeed, suppose $f_j=x_1^{a_1}\dots x_n^{a_n}z^{d-d_j}$, where $d-d_j>0$. Then we would get $\frac{f_j}zy_j$ among the generators of the last part. But when we take $z=1$, we get $f_jy_j$, which is a redundant generator as we already have $\frac{f_j}{x_k}y_j$ for some $k$, which divides $f_jy_j$.  In other words, when we take the quotient by $z-1$, the third summand gets absorbed in the second one.  And the second summand becomes 
$$\big(\frac{f_jy_j}{x_k}\mid j=1,\dots,m, \text{where $x_k$ divides $f_j$}\big)$$
when taking $z=1$.
So the only thing left is to describe the complete part $C$ of $\Lin(I^\eq)$ after taking $z=1$, let's call it $\cl C$. 

Among the generators of $C$ we have those where $z$~ has exponent~$d-\delta$, the highest possible. Then, when taking $z=1$, from these we get  all possible monomials of degree $\delta$ with exponent vector below~$v$. All the rest of the monomials in $\cl C$ (which of higher degree) are divided by such monomials. In short,
$$\cl C=(x_1,\dots,x_n)^\delta_{\le v}.$$
Observe that we don't necessarily have that all monomials in $(x_1,\dots,x_n)^\delta$ have exponent vector below $v$. Take for instance $I=(x_1^3,x_1x_2)$, so that $d=3$, $\delta=2$ and $v=(3,1)$. We have $I^\eq=(x_1^3,x_1x_2z)$ and $v^\eq=(3,1,1)$. For what concerns the complete part, we have $C=(x_1^3,x_1^2x_2,x_1^2z,x_1x_2z)$ and $\cl C=(x_1^2,x_1x_2)$.

So now the last question is: what is the interplay of $\cl C$ with the other part that survives, namely the second summand? We have the generators
$$h_{j,k}:=\frac{g_jy_j}{x_k}=\frac{f_jz^{d-d_j}y_j}{x_k},\qquad\text{where $x_k$ divides $f_j$.}$$
We have $\deg h_{j,k}=d$ for all $j$ and $k$. When taking $z=1$, for the residue class $\cl{h_{j,k}}$ we get $\deg(\cl{h_{j,k}})=d_j$. So at the end of the day the only survivors are those that come from those $f_j$'s of degree $\delta$. Otherwise, if $d_j>\delta$, then $\cl{h_{j,k}}=\frac{f_jy_j}{x_k}$ is divided by some monomial in $\cl C=(x_1,\dots,x_n)^\delta_{\le v}$. Because of course the vector of exponents of $f_j/x_k$ is below~$v$ and $\deg(f_j/x_k)=d_j-1\ge\delta$.

The discussion above amounts to a proof of the following.

\begin{corollary}
If we set $z=1$, the residue class of the ideal $\Lin(I^\eq)$ is
$$(x_1,\dots,x_n)^\delta_{\le v}
+\big(\frac{f_jy_j}{x_k}\mid j=1,\dots,m, \text{where $x_k$ divides $f_j$ and $d_j=\delta$}\big),$$
where $\delta=\min\{d_1,\dots,d_m\}$.
\end{corollary}

%\begin{remark}
%Marilina Rossi suggested that, instead of $I^\eq$, to generalize the linearization construction for an arbitrary monomial ideal $I$, we apply the usual $\Lin$ (for ideals generated in a single degree) to $I_{\langle d\rangle}$. In her opinion this is more canonical. My problem is that from $I^\eq$ we can easily get $I$ back, but from $I_{\langle d\rangle}$ we can't. For instance, start from an ideal $I$ generated in different degrees, then the truncation if degree $d$. If instead we take the truncation in degree $d$ of $I_{\langle d\rangle}$, we also get $I_{\langle d\rangle}$, but clearly $I_{\langle d\rangle}\ne I$. On the other hand, I wonder if one can reconstruct the ideal $I$ if they know the vector of highest exponents.
%\end{remark}

%%%%%%%%%%%%%%%%%%%%%%%%%%%%
%%%%%%%%%%%%%%%%%%%%%%%%%%%%

\section{Possible future directions}\label{domandeaperte}

Another way of defining the last part of $\Lin(I)$ in Definition~\ref{deflinearization} is saying that it's generated by the monomials $y_j\de f_j/\de x_k$. (One may even consider a ``monic'' partial derivative, as it sometimes happens.)  This gives a relaxation in the definition, in that one does not need to check that the variables with respect to which one differentiates actually are in the support of the monomials, because if not one simply gets zero. This definition could also provide a possible way to generalize the linearization construction to non-monomial ideals, %For instance, consider the ideal
%$$I:=(x_1x_2-x_1x_3)\subset\K[x_1,x_2,x_3].$$
%One may think of taking all the derivatives and define the linearization of $I$ as
%$$(x_1,x_2,x_3)^2+(x_1y,x_2y-x_3y)\subset\K[[x_1,x_2,x_3,y],$$
%and indeed this ideal has a linear resolution. 
a possibility that has not yet  been explored.

Generalizing the equification construction ``as it is'' to non-monomial ideals  seems to fail very easily, because the definition depends very much on the system of generators one considers. Even if one takes a homogeneous ideal, where the number of minimal homogeneous generators of each degree is invariant, the construction would still depend on the chosen system of generators: for instance, one has
$$(x+y,x^2)=(x+y,xy)\quad\text{but}\quad\big((x+y)z,x^2\big)\ne\big((x+y)z,xy\big).$$
It might be possible however to ``improve'' the definition of equification. For instance, one could fix a monomial order, and with respect to that monomial order each ideal has a canonical system of generators, its reduced Gr\"obner basis. In the monomial case this would reduce to our Definition~\ref{defequification}.

One more open problem is that of understanding whether there are any nice interplays of~$\Lin$ and standard operations on ideals, as briefly discussed in Section~\ref{linandops}.

%\subsection{Questions related to Betti splittings}\label{randombettistuff}
\medskip

When I had the chance to present the material in this paper, several questions were asked. A particularly interesting direction of investigation was suggested by Marilina Rossi. My knowledge of it relies mainly on the work of Bolognini in~\cite{bolo1}. There he studies the concept of \emph{Betti splitting}, already present in the literature. A different version of this notion appears already in particular in~\cite{HV}, which is interesting for us because of its relation to our Section~\ref{hyplinres}. Unfortunately we did not manage to find or prove anything fruitful concerning this topic, so only counterxamples are presented in this very last section.

\begin{definition}
Let $I$, $J$ and $K$ be monomial ideals such that $G(I)$ is the disjoint union of $G(J)$ and $G(K)$, so that in particular $I=J+K$. Then $I=J+K$ is a \tbs{Betti splitting} if
$$\beta_{i,j}(I)=\beta_{i,j}(J)+\beta_{i,j}(K)+\beta_{i-1,j}(J\cap K)$$
for all $i,j\in\N$.
\end{definition}

This concept is intimately related to ideals with linear resolutions or a generalization of them,   given by \emph{componentwise linear ideals} (see~\cite{HeHi} or Section~8.2 of~\cite{HH}). This is why the concept seemed naturally close to the topic of this paper. The definition of $\Lin(I)$ or $\LIN(I)$ already provides a very natural way of partitioning the generators, namely by choosing $J$ as the complete part~$C$ and $K$ as the last part~$L$. So one could reasonably expect this to be  a Betti splitting. But that's not the case in general, see the following  example.

\begin{example}
Consider $I=(x_1^3x_2,x_2x_3^3)\subset\K[x_1,x_2,x_3]$. Then we have
$$\beta(\Lin(I))=\begin{array}{c|ccc}
&0&1&2\\
\hline
4&11&16&6
\end{array}\qquad
\beta(C)=\begin{array}{c|ccc}
&0&1&2\\
\hline
4&7&8&2
\end{array}$$
$$\beta(L)=\begin{array}{c|ccc}
&0&1&2\\
\hline
4&4&2&-\\
5&-&-&-\\
6&-&1&-\\
7&-&1&1
\end{array}\qquad
\beta(C\cap L)=\begin{array}{c|ccc}
&0&1&2\\
\hline
5&6&4&-\\
6&-&1&-\\
7&-&1&1.
\end{array}$$
We do have
\begin{align*}
\beta_{0,1}(\Lin(I))&=11=4+7+0,\\
\beta_{1,5}(\Lin (I))&=16=8+2+6,\\
\beta_{2,6}(\Lin (I))&=6=2+0+4,
\end{align*}
which agree with the definition of Betti splitting, but unfortunately there is some more rubbish in $L$ and $C\cap L$ that would give something non-zero for $\Lin(I)$, whereas the Betti numbers of $\Lin(I)$ are zero. The same problems occurs considering $\LIN(I)$ instead of $\Lin(I)$.
\end{example}

\begin{question}
When do we have that $\Lin(I)$ (or $\LIN(I)$), decomposed as a sum of the complete part $C$ and the last part $L$, is a Betti splitting? Namely, when do we have that
$$\beta_{i,j}(\Lin(I))=\beta_{i,j}(C)+\beta_{i,j}(L)+\beta_{i-1,j}(C\cap L)$$
for all $i,j\in\N$? Or is there another meaningful way of partitioning the generators that constitutes a Betti splitting?
\end{question}

\begin{proposition}[Bolognini, Proposition 3.1 of~\cite{bolo1}]
Let $ I$ be a monomial ideal with a $d$-linear resolution,
and $J, K \ne 0$ monomial ideals such that $I = J + K$, $G(I) = G(J)\cup G(K)$ and $G(J)\cap G(K) = \emse$.
Then the following facts are equivalent:
\begin{itemize}
\item[(i)] $I = J + K$ is a Betti splitting of $I$;
\item[(ii)] $J$ and $K$ have $d$-linear resolutions.
\end{itemize}
If this is the case, then $J\cap K$ has a $(d + 1)$-linear resolution
\end{proposition}

Specializing this result to our notation, it means the following: $\Lin(I)=C+L$ is a Betti splitting if and only if  $C$ and $L$ have a linear resolution.
So, since we already know that $C$ has a linear resolution and $L$ is generated in degree~$d$, the result tells us that $\Lin(I)=C+L$ is a Betti splitting if and only if $L$ has  $d$-linear resolution (and $C\cap L$ is automatically $(d+1)$-linear). So  this motivates the following.

\begin{question}
When does $L$ have a linear resolution?
\end{question}

In particular, for $d=2$ we have Theorem~\ref{frofro} characterizing the squarefree quadratic monomials with linear resolution. 

\begin{example}
Consider the ideal $I=(x_1x_2,x_2x_3)\subset \K[x_1,x_2,x_3]$, which corresponds to the path on three vertices. This ideal has linear resolution. The last part of $\Lin(I)$ is $L=(x_1y_1,x_2y_1,x_2y_2,x_3y_2)$, and it doesn't have a linear resolution.
\end{example}

%\begin{remark}
%Notice a possibly interesting and useful thing: in the squarefree case of degree~2 (namely, Booth--Lueker) the graph corresponding to the last part $L$ is bipartite! This might help in understanding when the complement is chordal. I think it's not chordal if we manage to achieve a certain situation. First of all, since $L$ is bipartite, for $L$ we have a left complete subgraph and a right complete subgraph, and these two are joined by some edges here and there. The situation that we want to achieve is that we have $x_1$ and $x_2$ on the left and $y_{14}$, $y_{12}$, $y_{23}$ on the right, so that we have a cylce $x_1$---$x_2$---$y_{14}$---$y_{12}$---$y_{23}$---$x_1$ without chords. (Or maybe we could also have a symmetric situation, with two $y$'s on the right and three $x$'s on the left, but I think what I wrote before is easier to achieve if we start from a very specific graph, like complete graph, let's say.)
%\end{remark}

In general,  it would be interesting to find properties of the last part $L$ alone. Observe that one could define it for arbitrary monomial ideals in the same way as it is for equigenerated ideals, and investigate more in general properties of $L$ in that case. We now conclude with a generalization of the concept of ideal with linear resolution.

\begin{definition}
For a homogenous ideal $I\se S$ we denote $I_{\langle d\rangle}$ the ideal generated by all homogeneous elements of degree~$d$ in $I$. We say that $I$ is \tbs{componentwise linear} if $I_{\langle d\rangle}$ has a linear resolution for all $d$.
\end{definition}

Componentwise linear ideals were introduced by~{Herzog} and~{Hibi} in~\cite{HeHi}. Ideals with linear resolutions are componentwise linear and in particular ideals with linear quotients are componentwise linear (see~\cite{HH}, Lemma~8.2.10 and Theorem~8.2.15).
The ideals involved in this paper are equigenerated, so being componentwise linear for them forces having a linear resolution. But perhaps this could be a meaningful concept to analyze in case the linearization can be defined in such a way that it's not  necessarily equigenerated anymore. In particular, we quote one last result from~\cite{bolo1}.

\begin{theorem}[Bolognini, Theorem~3.3 of~\cite{bolo1}]
Let $I$, $J$ and $K$ be monomial ideals such that $I = J + K$ and $G(I)$ is the
disjoint union of~$G(J)$ and $G(K)$. If $J$ and $K$ are componentwise linear, then $I = J +K$ is a Betti splitting of $I$.
\end{theorem}


\begin{thebibliography}{35} 

\bibitem{CoCoA}
J. Abbott, A.M. Bigatti, L. Robbiano.
\emph{CoCoA: a system for doing Computations in Commutative Algebra.}
Available at \texttt{http://cocoa.dima.unige.it}

\bibitem{AFL}
A. Almousa, G. Fl\o ystad, H. Lohne.
Polarizations of powers of graded maximal ideals. arXiv:1912.03898.


%\bibitem{BS08}
%Mats Boij and Jonas S\"oderberg.
%Graded Betti numbers of Cohen-Macaulay modules and the multiplicity conjecture. 
%\emph{J. Lond. Math. Soc. (2)} {\bf 78} (2008), no. 1, 85--106. 

\bibitem{bolo1}
D. Bolognini.
Betti splitting via componentwise linear ideals. \emph{J. Algebra} \textbf{455} (2016), 1--13.

%\bibitem{bolo2}
%Davide Bolognini.
%Secondo articolo sul Betti splitting. Vedi cartella dove salvo il manoscritto. Questo non lo uso, forse, per il momento lo menziono solo.

\bibitem{BL75}
K.S. Booth, G.S. Lueker.
Linear algorithms to recognize interval graphs and test for the consecutive ones property. \emph{Seventh Annual ACM Symposium on Theory of Computing (Albuquerque, N. M., 1975)}, Assoc. Comput. Mach., New York, 1975, pp. 255--265.

\bibitem{BrCo}
W. Bruns, A. Conca.
Linear resolutions of powers and products. \emph{Singularities and computer algebra}, 47--69, Springer, Cham, 2017.

\bibitem{BH93}
W. Bruns, J. Herzog.
\emph{Cohen--Macauly rings. Second edition.} Cambridge Studies in Advanced Mathematics, 39. Cambridge University Press, Cambridge, 1993. xiv+453.

\bibitem{CoHe}
A. Conca; J. Herzog.
Castelnuovo--Mumford regularity of products of ideals.  \emph{Collect. Math.} \textbf{54} (2003), no. 2, 137--152.

\bibitem{CLO}
D. Cox, J. Little, D. O'Shea. Ideals, varieties, and algorithms. An introduction to computational algebraic geometry and commutative algebra. Undergraduate Texts in Mathematics. \emph{Springer-Verlag, New York,} 1992. xii+513 pp.

%\bibitem{DE09}
%Anton Dochtermann and Alexander Engstr\"om.
%Algebraic properties of edge ideals via combinatorial topology.
%\emph{Electron. J. Combin.} {\bf 16} (2009), no. 2, Special volume in honor of Anders Bj\"orner, Research Paper 2, 24~ pp. 

\bibitem{EMO}
J. Eagon, E. Miller, E. Ordog.
Minimal resolutions of monomial ideals. arXiv:1906.08837.

\bibitem{EaRe}
J. Eagon, V. Reiner. 
\emph{Resolutions of Stanley--Reisner rings and Alexander duality},
Journal of Pure and Applied Algebra \textbf{130} (1998), no.~3, 265--275.

\bibitem{EisCA}
D. Eisenbud.
Commutative algebra. With a view toward algebraic geometry. Graduate Texts in Mathematics, 150. \emph{Springer-Verlag, New York,} 1995. xvi+785 pp.

\bibitem{EiGo}
D. Eisenbud, S. Goto.
Linear free resolutions and minimal multiplicity, \emph{J. Algebra} \textbf{88} (1984), no. 1, 89--133.

%\bibitem{ES11}
%David Eisenbud and Frank-Olaf Schreyer.
%Betti numbers of graded modules and cohomology of vector bundles. 
%\emph{J. Amer. Math. Soc.} {\bf 22} (2009), no. 3, 859--888. 

\bibitem{EJO}
A. Engstr\"om, L. Jakobsson, M. Orlich.
Explicit Boij--S\"oderberg theory of ideals from a graph isomorphism reduction. \emph{J. Pure Appl. Algebra} \textbf{224} (2020), no. 11, 106405, 17 pp.

%\bibitem{ES13}
%Alexander Engstr\"om and Matthew T. Stamps.
%Betti diagrams from graphs.
%\emph{Algebra and Number Theory} {\bf 7} (2013), no. 7, 1725--1742. 

%\bibitem{FLBS}
%Gunnar Fl\o ystad. 
%Boij--S\"oderberg theory: introduction and survey. 
%\emph{Progress in commutative algebra 1}, 1--54, de Gruyter, Berlin, 2012.

\bibitem{F90}
R. Fr\"oberg.
\emph{On Stanley--Reisner rings.} Topics in algebra, Part 2 (Warsaw, 1988), 57--70, 
Banach Center Publ., \textbf{26}, Part 2, PWN, Warsaw, 1990. 

\bibitem{HV}
H.T. H\`a, A. Van Tuyl.
Monomial ideals, edge ideals of hypergraphs, and their graded Betti numbers. 
\emph{J. Algebraic Combin.} {\bf27} (2008), no. 2, 215--245.

\bibitem{Hart}
R. Hartshorne, 
\emph{Connectedness of the hilbert scheme}, Publications Mathématiques de l'IH\'ES \textbf{29} (1966), 5--48.

\bibitem{HeHi}
J. Herzog, T. Hibi.
\emph{Componentwise linear ideals,} Nagoya Math. J. \textbf{153} (1999), pp. 141--153.

\bibitem{HH}
J. Herzog, T. Hibi
Monomial ideals. Graduate Texts in Mathematics, 260. \emph{Springer-Verlag London, Ltd., London,} 2011. xvi+305 pp.

\bibitem{HeTa}
J. Herzog, Y. Takayama. 
Resolutions by mapping cones, \emph{Homology Homotopy Appl}. \textbf{4} (2002), 277–294.

%\bibitem{Hoch}
%Melvin Hochster.
%stanley reister stuff (see Peeva)

\bibitem{K73}
D. Knuth.
\emph{The Art of Computer Programming, Vol. 1, Fundamental Algorithms,} 2nd ed., Addison-Wesley, Reading, Mass., 1973.

%\bibitem{Ma}
%Manolis.
%\emph{Il suo articolo con robe su risoluzioni lineari.}

\bibitem{Lyub}
G. Lyubeznik.
A new explicit finite free resolution for ideals generated by monomials in an~$R$-sequence. \emph{J.~{Pure} Appl.~{Algebra}} \textbf{51} (1988), 193--195.

%\bibitem{MaPe}
%Jason McCullugh and Irena Peeva.
%Regularity conjecture.

\bibitem{MiSt}
E. Miller, B. Sturmfels.
Combinatorial commutative algebra.
Graduate Texts in Mathematics, 227. \emph{Springer-Verlag, New York,} 2005. xiv+417 pp.

\bibitem{Nov}
I. Novik.
Lyubeznik's resolution and rooted complexes, \emph{J.~{Alg.}~{Combin.}} \textbf{16} (2002), 97--101.

\bibitem{Pe}
I. Peeva.
Graded syzygies. Algebra and Applications, 14. \emph{Springer-Verlag London, Ltd., London,} 2011. xii+302 pp.

\bibitem{Rei}
G.A. Reisner. 
Cohen--Macaulay quotients of polynomial rings. \emph{Advances in Math.} \textbf{21} (1976), no. 1, 30--49.

\bibitem{Stan}
R. Stanley.
Combinatorics and commutative algebra.  Progress in Mathematics, 41. \emph{Birkh\"auser Boston, Inc., Boston, MA,} 1983. viii+88 pp.

\bibitem{Steu}
M. Steurich. 
On rings with linear resolutions. \emph{J. Algebra} \textbf{75} (1982), no. 1, 178–197.

\bibitem{Tay}
D. Taylor.
Ideals generated by monomials in an $R$-sequence, \emph{Ph. D. Thesis,} University of Chicago, 1966.

\bibitem{Yana}
K. Yanagawa. 
Alternative polarizations of Borel fixed ideals. \emph{Nagoya Mathematical Journal} \textbf{207} (2012), 79--93.





\end{thebibliography}
\end{document}